\documentclass[reqno,12pt]{amsart}
\usepackage{amsmath,amsfonts,amsthm,amssymb,indentfirst}
\usepackage[all,poly]{xy} 
\usepackage{color}
\usepackage[pagebackref,colorlinks]{hyperref}
\usepackage{kantlipsum}
\usepackage[all,poly]{xy}
\usepackage{mathrsfs}
\usepackage{enumerate}

\theoremstyle{plain}
\newtheorem{lemma}{Lemma}[section] 
\newtheorem{theorem}[lemma]{Theorem}
\newtheorem{corollary}[lemma]{Corollary}
\newtheorem{proposition}[lemma]{Proposition}

\theoremstyle{definition}
\newtheorem{remark}[lemma]{Remark}
\newtheorem{example}[lemma]{Example}

\newtheorem{question}[lemma]{Question}

\newcommand{\Zset}{\mathbb Z}

\newcommand{\Cset}{\mathbb C}

\newcommand{\M}{\operatorname{\mathbb M}}
\newcommand{\gr}{\operatorname{gr}}
\newcommand{\ol}{\overline}
\newcommand{\w}{\operatorname{w}}
\newcommand{\s}{\operatorname{s}}

\newcommand{\End}{\operatorname{End}}
\newcommand{\END}{\operatorname{END}}
\newcommand{\Hom}{\operatorname{Hom}}
\newcommand{\HOM}{\operatorname{HOM}}
\newcommand{\coker}{\operatorname{coker}}
\newcommand{\sr}{\operatorname{sr}}

\newcommand{\so}{\mathbf{s}}
\newcommand{\ra}{\mathbf{r}}

\newcommand{\V}{\mathcal V}
\newcommand{\Proj}{\mathcal P}
\newcommand{\Mod}{\mathcal M}

\DeclareRobustCommand{\subtitle}

\title[Graded cancellation properties of graded rings]{Graded cancellation properties of graded rings and graded unit-regular Leavitt path algebras}

\author{Lia Va\v s}
\address{Department of Mathematics, Physics and Statistics, University of the Sciences, Philadelphia, PA 19104, USA}
\email{l.vas@usciences.edu}


\thanks{A part of this paper was conceived during the Asia-Australia Algebra Conference and the author's visit to Western Sydney University in January 2019. This visit was supported by the NSF-AWM Travel Grants program.}

\begin{document}

\begin{abstract} We raise the following general question regarding a ring graded by a group: ``If $P$ is a ring-theoretic property, how does one define the graded version $P_{\operatorname{gr}}$ of the property $P$ in a meaningful way?''. Some properties of rings have straightforward and unambiguous generalizations to their graded versions and these generalizations satisfy all the matching properties of the nongraded case. If $P$ is either being unit-regular, having stable range 1 or being  directly finite, that is not the case. The first part of the paper addresses this issue. Searching for appropriate generalizations, we consider graded versions of cancellation, internal cancellation, substitution, and module-theoretic direct finiteness.  

In the second part of the paper, we consider graded matrix and Leavitt path algebras. If $K$ is a trivially graded field and $E$ is a directed graph, the Leavitt path algebra $L_K(E)$ is naturally graded by the ring of integers. If $E$ is a finite graph, we present a property of $E$ which is equivalent with $L_K(E)$ being graded unit-regular. This property critically depends on the lengths of paths to cycles and it further illustrates that graded unit-regularity is quite restrictive in comparison to the alternative generalization of unit-regularity from the first part of the paper. 
\end{abstract}

\maketitle

\section{Introduction}

We raise the question ``If $P$ is a ring-theoretic property, how does one define the graded version $P_{\gr}$ of the property $P$ in a meaningful way?'' and consider it in the cases when $P$ is unit-regularity, cancellability, stable range 1, and direct finiteness. To address these cases, we study graded generalizations of some of cancellation properties from T.Y. Lam's ``A crash course on stable range, cancellation, substitution and exchange'' (\cite{Lam_cancellation_properties}). We focus on these properties since it is not as obvious and straightforward to define their graded generalization as it is for some other properties and we elaborate on this in the introduction. We do not assume that the grade group $\Gamma$ is abelian and study some properties previously considered only for abelian groups $\Gamma.$ In the last part of the paper, we consider these properties for Leavitt path algebras of finite graphs and characterize them in terms of the properties of the graph. 

A ring $R$ is graded by a group $\Gamma$ if $R=\bigoplus_{\gamma\in\Gamma} R_\gamma$ for additive subgroups $R_\gamma$ and $R_\gamma R_\delta\subseteq R_{\gamma\delta}$ for all $\gamma,\delta\in\Gamma.$ The elements of the set $H=\bigcup_{\gamma\in\Gamma} R_\gamma$ are said to be homogeneous. The grading is trivial if $R_\gamma=0$ for every nonidentity $\gamma\in \Gamma.$ Since every ring is graded by the trivial group, we can say that the class of graded rings generalizes the class of rings. Still, it is customary that a ring graded by the trivial group is referred to as a nongraded ring. 

If a ring-theoretic property $P$ is in its prenex form, the term {\em graded property $P$} has been used for the property  
$P_{\gr}$ obtained by replacing every $\forall x$ and $\exists x$ appearing in $P$ by the restricted versions $\forall x\in H$ and $\exists x\in H.$  For example, if Reg is the property  $(\forall x)(\exists y)(xyx=x)$ defining a von Neumann regular (or regular for short), we say that a graded ring $R$ is graded regular if $(\forall x\in H)(\exists y\in H)(xyx=x)$ and we denote this condition by Reg$_{\gr}.$

If a property $P$ has the form $(\forall x)(\exists y)\phi(x, y),$ let $P^{\w}_{\gr}$ denote the statement $(\forall x\in H)(\exists y)\phi(x, y)$ which we call {\em the weak graded property $P$}. In some cases,  $P^{\w}_{\gr}$ and $P_{\gr}$ are equivalent. For example, if $P$ is the property Reg, then $P^{\w}_{\gr}$ and $P_{\gr}$ are equivalent (if $x=xyx$ for $x\in R_\gamma,$ and if $y_{\gamma^{-1}}$ is the $\gamma^{-1}$-component of $y,$ then $xy_{\gamma^{-1}}x=x$).     

The situation is trickier if $P$ is a ring-theoretic property such that $P^{\w}_{\gr}$ and $P_{\gr}$ are not equivalent. For example, let UR be  the property  $(\forall x)(\exists u)(\exists v)(uv=vu=1$ and $x=xux$) defining a unit-regular ring. The conditions UR$^{\w}_{\gr}$ and UR$_{\gr}$ are not equivalent. Indeed, let $K$ be a field trivially graded by the group of integers $\Zset$ and $R$ be the graded matrix ring $\M_2(K)(0,1)$ (we review the definition in Section \ref{subsection_prerequisites}). The standard matrix unit $e_{12}$ is homogeneous. As a homogeneous invertible element $u$ is diagonal, $e_{12}ue_{12}=e_{12}$ for no such $u$ and so UR$_{\gr}$ fails. On the other hand, $\M_2(K)(0,1)$ satisfies UR$^{\w}_{\gr}$ since $\M_2(K)$ is unit-regular. The graded ring $\M_2(K)(0,1)$ is also an example of a ring which is graded semisimple but not graded unit-regular. This shows that if $P\Rightarrow Q$ holds for some properties of rings, then it may happen that $P_{\gr}\nRightarrow Q_{\gr}.$ Moreover, if a property $P$ has a feature $F,$ then the graded version $F_{\gr}$ may fail to hold for $P_{\gr}.$ For example, while UR is closed under formation of matrix algebras and corners, UR$_{\gr}$ is not closed under formation of graded matrix algebras (by the example with $\M_2(K)(0,1)$ above) and graded corners (by Example \ref{example_corners}). 
  
The above discussion seems to indicate that more than one aspect should be taken into consideration if looking for a meaningful way to generalize properties to graded rings. 
In some cases, a ring-theoretic definition may just be a convenient simplification of certain equivalent module-theoretic property. Sometimes the historical origin of a definition may provide a meaningful insight in the process of a generalization to graded rings. Considering all of these factors, we ask the following question, central for the motivation of the work in the first part of this paper:
\begin{question}
If $P$ is a ring-theoretic property, how does one define the graded version $P_{\gr}$ of the property $P$ in a meaningful way?
\end{question}
Unit-regularity, for example, originated as {\em a property of the endomorphism ring of a module, not the ring itself}. In the graded case, the graded component of the endomorphism ring corresponding to the group identity $\epsilon\in\Gamma$ has a special significance. Requiring this component of a graded ring to be unit-regular brings us to far less restrictive requirement UR$_\epsilon$ than UR$_{\gr}$ and we show that the condition Reg$_{\gr}$+UR$_\epsilon$ is less restrictive than UR$_{\gr}$ but still strong enough to capture the relevant properties of unit-regularity in the graded case.  

After a review of prerequisites and some preliminary results in Section \ref{section_prerequisites_preliminaries}, we consider graded versions of module-theoretic characterizations of unit-regularity in Section \ref{section_unit_regular_and_cancellability}. If $P(A)$ is a property of a module $A,$ we let $P_{\gr}(A)$ denote the statement obtained from $P(A)$ if every instance of ``module'' in it is replaced by ``graded module'' and every instance of ``homomorphism'' by ``graded homomorphism''. In particular, every isomorphism $\cong$ is replaced by a graded isomorphism $\cong_{\gr}.$ 

For nongraded rings, a module $A$ has {\em internal cancellation} IC$(A)$ if 
\begin{itemize}
\item[IC($A$):] \hskip.5cm $A=B\oplus C=D\oplus E$ and $B\cong D$ implies $C\cong E$
\end{itemize}
holds for all modules $B,C,D,E$ in which case we also say that $A$ is internally cancellable. This property is equivalent with UR of $S=\End_R(A)$ if $S$ is regular. If $A$ is a graded module, the subring $\END_R(A)$ of $\End_R(A),$ generated by graded homomorphisms of degree $\gamma$ for all $\gamma\in\Gamma,$ is naturally graded and the elements of $\END_R(A)_\epsilon$ are exactly the {\em graded} endomorphisms of $A.$ Thus, the statement IC$_{\gr}(A),$ the graded version of IC($A$) obtained by the process we explained above, translates to a property of $\END_R(A)_\epsilon$ only, not the entire ring $\End_R(A).$ If $\END_R(A)_\epsilon$ is regular, we show that IC$_{\gr}(A)$ holds if and only if $\END_R(A)_\epsilon$ is unit-regular (Proposition \ref{graded_ur_of_end_ring}). In addition, recall that an $R$-module $A$ is {\em cancellable} in a category of $R$-modules $\Mod$ if the condition
\begin{itemize}
\item[C($A$):] \hskip.5cm $A\oplus B\cong A\oplus C$ implies $B\cong C$
\end{itemize}
holds for all $B$ and $C$ in $\Mod.$ If $R$ is regular, then $R$ has UR if and only if C$(R)$ holds in the category of finitely generated projective modules. If $R$ is a graded ring, $A$ a graded module, and $\Mod_{\gr}$ a category of graded modules, $A$ is {\em graded cancellable} in $\Mod_{\gr}$ if C$_{\gr}(A)$ holds in $\Mod_{\gr}.$ If $\Proj_{\gr}$ denotes the category of finitely generated graded projective modules, we say that C$_{\gr}$ holds if C$_{\gr}(P)$ holds for every $P$ in $\Proj_{\gr}$ and that $R$ has {\em graded cancellation} in this case. We also say that IC$_{\gr}$ holds if IC$_{\gr}(P)$ holds for every $P$ in $\Proj_{\gr}$ and that $R$ has {\em graded internal cancellation} in this case. By Theorem \ref{C_and_IC},
if Reg$_{\gr}$ holds, then UR$_\epsilon,$ IC$_{\gr},$ C$_{\gr}(R)$ and C$_{\gr}$ are all equivalent. 
The condition Reg$_{\gr}$+UR$_\epsilon,$ formulated without referring to any module, is far less restrictive than UR$_{\gr}$ but strong enough to guarantee that $C_{\gr}$ holds. By Corollary \ref{morita_invariant}, Reg$_{\gr}$+UR$_\epsilon$ is graded Morita invariant and, by the example with $\M_2(K)(0,1),$ UR$_{\gr}$ is not. So, Reg$_{\gr}$+UR$_\epsilon$ also avoids this anomaly of UR$_{\gr}.$    

We compare UR$_{\gr}$ and Reg$_{\gr}+$ UR$_\epsilon$ in one more aspect. Namely, Handelman's Conjecture stipulates that every $*$-regular ring is unit-regular. In the graded case, it is not difficult to find an example of a graded $*$-regular ring which is not graded unit-regular (for example, $\M_2(\Cset)(0,1)$ with $\Cset$ trivially graded by $\Zset$ and $*$ induced by the complex-conjugation). Since a graded $*$-regular ring is such that $R_\epsilon$ is $*$-regular and Reg$_{\gr}$ holds, we believe that it is more meaningful to ask whether UR$_\epsilon$ also holds than to come up with an example showing that UR$_{\gr}$ fails. So, we ask the following.  
\begin{question}
Is the $\epsilon$-component of every graded $*$-regular ring unit-regular?
\label{question_graded_Handelman}
\end{question}
In Section \ref{subsection_Handelman}, we note that the answer to this  question is ``yes'' for unital Leavitt path algebras.  

In Section \ref{section_unit_regular_and_cancellability}, we also consider the weak and strong graded internal cancellation properties, IC$^{\w}_{\gr}(R)$ and IC$^{\s}_{\gr}(R)$ (Proposition \ref{graded_unit_regular}). Although weaker than IC$^{\s}_{\gr}(\underline{\hskip.3cm}),$ IC$^{\w}_{\gr}(\underline{\hskip.3cm})$ ends up being outside of the category of graded modules while  Reg$_{\gr}$+UR$_\epsilon$ does not have this disadvantage. 

In Section \ref{section_graded_sr_1_and_DF}, we consider the stable range 1 property. Let $\sr(R)=1$ stand for $(\forall x, y)(\exists z,u) (xR+yR=R \Rightarrow$ $z=x+yu$ and $zR=R)$ and let $\sr_{\gr}(R)=1$ denote the graded version of $\sr(R)=1.$ We have that UR$_{\gr}(R)\;$ $\Rightarrow$ $\sr_{\gr}(R)=1\;$ $\Rightarrow$ $\sr(R_\epsilon)=1\;$ $\Rightarrow$ C$_{\gr}(R)$ and we show that each implication is strict. 
The property $\sr(R)=1$ also has its module-theoretic characterization related to cancellation. If $A$ is an $R$-module, $\sr(\End_R(A))=1$ is equivalent with the property below, known as {\em substitution}.  
\begin{itemize}
\item[S$(A):$] If $A\oplus B=A'\oplus B'=M$ for some modules $M, A', B, B'$ and $A\cong A',$ then there is a module $C$ such that  $B\oplus C=B'\oplus C=M.$
\end{itemize}
Let S$_{\gr}(A)$ denote the graded version of this property. By Theorem  \ref{substitution_and_sr}, S$_{\gr}(A)$ holds for a graded module $A$ if and only if $\sr(\END_R(A)_\epsilon)=1.$ By Theorem \ref{substitution_and_sr}, $\sr(\END_R(A)_\epsilon)=1,$ a much weaker condition that $\sr_{\gr}(\END_R(A))=1,$ is sufficient for $A$ being graded cancellable without any additional requirements on $A.$ This shows that the conclusion of the Graded Cancellation Theorem (\cite[Theorem 1.8.4]{Roozbeh_book}) holds without requiring $A$ to be finitely generated and $\Gamma$ to be abelian
and with the assumption $\sr_{\gr}(\END_R(A))=1$ replaced by $\sr(\END_R(A)_\epsilon)=1.$  

In Section \ref{subsection_graded_finiteness}, we consider direct finiteness DF, closely related to other cancellability conditions, its module-theoretic version DF($\underline{\hskip.3cm}$), and their graded generalizations DF$_{\gr},$ DF$^{\w}_{\gr}(\underline{\hskip.3cm}),$ and DF$^{\s}_{\gr}(\underline{\hskip.3cm}).$ The diagram in \cite[Formula (4.2)]{Lam_cancellation_properties} becomes S$_{\gr}(\underline{\hskip.3cm})\;$ $\Rightarrow\;$ C$_{\gr}(\underline{\hskip.3cm})\;$  $\Rightarrow\;$  IC$_{\gr}(\underline{\hskip.3cm})$  $\Rightarrow\;$  DF$_{\gr}(\underline{\hskip.3cm}).$

In Section \ref{subsection_strongly_graded}, we show that if $R$ is strongly graded,  
S$_{\gr}(A)$ $\Leftrightarrow$ S$(A_\epsilon),\,$ 
C$_{\gr}(A)$ $\Leftrightarrow$ C$(A_\epsilon),\,$ 
IC$_{\gr}(A)$ $\Leftrightarrow$ IC$(A_\epsilon),\,$ and
DF$_{\gr}(A)$ $\Leftrightarrow$ DF$(A_\epsilon),\,$ 
for a graded $R$-module $A$ (Proposition \ref{strongly_graded_cancellability}). In contrast, we show that the following three implications are strict even if $R$ is strongly graded: 
$R$ satisfies UR$_{\gr}$ $\Rightarrow$ $R_\epsilon$ satisfies UR,  $\sr_{\gr}(R)=1\Rightarrow \sr(R_\epsilon)=1,$ 
$R$ satisfies DF$_{\gr}$ $\Rightarrow$ $R_\epsilon$ satisfies DF.

In Section \ref{section_matrices_and_LPAs}, we produce a condition on shifts which characterizes when a $\Zset$-graded matrix algebra over a trivially graded field $K$ or over naturally $\Zset$-graded $K[x^m,x^{-m}]$ is graded unit-regular (Proposition \ref{proposition_matrices}). 
If $E$ is a finite graph and $L_K(E)$ is the Leavitt path algebra of $E$ over a field $K$, Theorem \ref{finite_graphs_ur_characterization} characterizes graded unit-regularity of $L_K(E)$ in terms of a property of $E.$ This property critically depends on the lengths of paths to cycles and it further illustrates that UR$_{\gr}$ is quite restrictive in comparison to Reg$_{\gr}$+UR$_\epsilon$. Proposition \ref{LPA_characterizations} characterizes other cancellability properties considered in this paper for $L_K(E).$

\section{Graded rings prerequisites and some preliminaries}\label{section_prerequisites_preliminaries}

Unless stated otherwise, $\Gamma$ denotes an arbitrary group, not necessarily abelian, and $\epsilon$ denotes its identity element. Rings are assumed to be associative. Unless stated otherwise, rings are assumed to be unital and a module is assumed to be a right module. 

\subsection{Graded rings prerequisites}\label{subsection_prerequisites}
In the introduction, we recalled the definitions of a graded ring, homogeneous elements and trivial grading. We adopt the standard definitions of graded ring homomorphisms and isomorphisms, graded left and right $R$-modules, graded module homomorphisms, graded algebras, graded left and right ideals, graded left and right free and projective modules as defined in \cite{NvO_book} and \cite{Roozbeh_book}. We recall that a $\Gamma$-graded
ring $R$ is a graded division ring if every nonzero element of $R$ has a multiplicative inverse. In this case, $R$ is a graded field if $R$ is also commutative.

If $M$ is a graded right $R$-module and $\gamma\in\Gamma,$ the $\gamma$-\emph{shifted or $\gamma$-suspended} graded right $R$-module $(\gamma)M$ is defined as the module $M$ with the $\Gamma$-grading given by $(\gamma)M_\delta = M_{\gamma\delta}$ for all $\delta\in \Gamma.$ 
Analogously, if $M$ is a graded left $R$-module, the $\gamma$-shifted  left $R$-module $M(\gamma)$ is the module $M$ with the $\Gamma$-grading given by $M(\gamma)_\delta = M_{\delta\gamma}$ for all $\delta\in \Gamma.$
Any finitely generated graded free right $R$-module is of the form $(\gamma_1)R\oplus\ldots\oplus (\gamma_n)R$ for $\gamma_1, \ldots,\gamma_n\in\Gamma$ and an analogous statement holds for finitely generated graded free left $R$-modules (both \cite{NvO_book} and \cite{Roozbeh_book} contain details).  

If $M$ and $N$ are graded right $R$-modules and $\gamma\in\Gamma$, then $\Hom_R(M,N)_\gamma$ denotes the following 
$$\Hom_R(M,N)_\gamma=\{f\in \Hom_R(M, N)\,|\,f(M_\delta)\subseteq N_{\gamma\delta}\},$$ 
then the subgroups $ \Hom_R(M,N)_\gamma$ of $\Hom_R(M,N)$ intersect trivially and $\HOM_R(M,N)$ denotes their direct sum $\bigoplus_{\gamma\in\Gamma} \Hom_R(M,N)_\gamma.$ The notation $\END_R(M)$ is used in the case if $M=N.$ If $M$ is finitely generated (which is the case we often consider), then $\Hom_R(M,N)=\HOM_R(M,N)$ for any $N$ (both \cite{NvO_book} and \cite{Roozbeh_book} contain details) and $\End_R(M)=\END_R(M,M)$ is a $\Gamma$-graded ring.

In  \cite{Roozbeh_book}, for a $\Gamma$-graded ring $R$ and $\gamma_1,\dots,\gamma_n\in \Gamma$, $\M_n(R)(\gamma_1,\dots,\gamma_n)$ denotes the ring of matrices $\M_n(R)$ with the $\Gamma$-grading given by  
\begin{center}
$(r_{ij})\in\M_n(R)(\gamma_1,\dots,\gamma_n)_\delta\;\;$ if $\;\;r_{ij}\in R_{\gamma_i^{-1}\delta\gamma_j}$ for $i,j=1,\ldots, n.$ 
\end{center}
The definition of $\M_n(R)(\gamma_1,\dots,\gamma_n)$ in \cite{NvO_book} is different: $\M_n(R)(\gamma_1,\dots,\gamma_n)$ in \cite{NvO_book} corresponds to $\M_n(R)(\gamma_1^{-1},\dots,\gamma_n^{-1})$ in \cite{Roozbeh_book}. More details on the relations between the two definitions
can be found in \cite[Section 1]{Lia_realization}. Although the definition 
from \cite{NvO_book} has been in circulation longer, some matricial representations of Leavitt path algebras involve positive integers instead of negative integers making the definition from \cite{Roozbeh_book} more convenient for us. Since we deal extensively with Leavitt path algebras in sections \ref{subsection_LPA_ur} and \ref{subsection_LPA_other}, we opt to use the definition from \cite{Roozbeh_book}. With this definition, if $F$ is the graded free right module $(\gamma_1^{-1})R\oplus \dots \oplus (\gamma_n^{-1})R,$ then $\Hom_R(F,F)\cong_{\gr} \;\M_n(R)(\gamma_1,\dots,\gamma_n)$ as graded rings.    

We also recall \cite[Remark 2.10.6]{NvO_book} stating the first two parts in Lemma \ref{lemma_on_shifts} and \cite[Theorem 1.3.3]{Roozbeh_book}  stating part (3) for $\Gamma$ abelian. The proof of this statement generalizes to arbitrary $\Gamma.$ 

\begin{lemma}\cite[Remark 2.10.6]{NvO_book}, \cite[Theorem 1.3.3]{Roozbeh_book}.
Let $R$ be a $\Gamma$-graded ring and $\gamma_1,\ldots,\gamma_n\in \Gamma.$ 
\begin{enumerate}
\item If $\pi$ a permutation of the set $\{1,\ldots, n\},$ then 
\begin{center}
$\M_n (R)(\gamma_1, \gamma_2,\ldots, \gamma_n)\;\cong_{\gr}\;\M_n (R)(\gamma_{\pi(1)}, \gamma_{\pi(2)} \ldots, \gamma_{\pi(n)}).$
\end{center}
 
\item If $\delta$ in the center of $\Gamma,$ $\;\M_n (R)(\gamma_1, \gamma_2, \ldots, \gamma_n)\;=\;\M_n (R)(\gamma_1\delta, \gamma_2\delta,\ldots, \gamma_n\delta).$

\item If $\delta\in\Gamma$ is such that there is an invertible element $u_\delta$ in $R_\delta,$ then  
\begin{center}
$\M_n (R)(\gamma_1, \gamma_2, \ldots, \gamma_n)\;\cong_{\gr}\;\M_n (R)(\gamma_1\delta, \gamma_2\ldots, \gamma_n).$
\end{center}
\end{enumerate}
\label{lemma_on_shifts} 
\end{lemma}

\subsection{Four lemmas}

Recall that for two idempotents $e,f$ of a ring $R,$  $eR\cong fR$ holds if and only if there are $x,y\in R$ such that $xy=e$ and $yx=f$ in which case we write $e\sim f$ and we can require that $x\in eRf$ and $y\in fRe$. This equivalence can be used to show that $eR\cong fR$ holds if and only if $Re\cong Rf$ holds (see \cite[Proposition 5.2]{Berberian_web}). The following lemma, needed for Proposition \ref{graded_unit_regular} shows the graded version of these equivalences. Note that if $R$ is graded and $e$ is a homogeneous idempotent in $R_\gamma$ then the relation $ee=e$ implies that $\gamma=\epsilon.$

\begin{lemma}
Let $R$ be a $\Gamma$-graded ring and $e,f$ homogeneous idempotents of $R$. The following conditions are equivalent. 
\begin{enumerate}
\item $eR\cong_{\gr} (\gamma)fR$ for some $\gamma\in\Gamma.$
 
\item $Re\cong_{\gr} Rf(\gamma^{-1})$ for some $\gamma\in\Gamma.$

\item There is $x\in R_{\gamma^{-1}}$ and $y\in R_{\gamma}$ such that $xy=e$ and $yx=f.$ 

\item There is $x\in eR_{\gamma^{-1}}f$ and $y\in fR_{\gamma}e$ such that $xy=e$ and $yx=f.$
\end{enumerate}
\label{graded_equivalence}
\end{lemma}
\begin{proof}
(1) $\Rightarrow$ (4). If $\phi:$ $eR\cong_{\gr} (\gamma)fR,$ then $\phi\in \Hom_R(eR, (\gamma)fR)_\epsilon.$ Thus, $y=\phi(e)\in (\gamma)fR_\epsilon\subseteq R_\gamma.$ Analogously, $x=\phi^{-1}(f)\in eR_{\gamma^{-1}}\subseteq R_{\gamma^{-1}}.$ Moreover, $ye=\phi(e)e=\phi(ee)=\phi(e)=y$ so $y\in Re$  and $x\in Rf$ similarly. Then $yx=\phi(e)x=\phi(ex)=\phi(x)=\phi(\phi^{-1}(f))=f$ and $xy=e$ similarly.

(4) $\Rightarrow$ (1). If $L_x$ and $L_y$ denote the left multiplications by $x$ and $y$ respectively, then $L_y\in \Hom_R(R, R)_\gamma=\Hom_R(R, (\gamma)R)_\epsilon$ and, similarly, $L_x\in  \Hom_R(R, R)_{\gamma^{-1}}=\Hom_R((\gamma)R, R)_\epsilon.$ The conditions $x\in eRf$ and $y\in fRe$ imply that $L_y$ maps $eR$ into $(\gamma)fR$ and $L_x$ maps $(\gamma)fR$ into $eR.$ 
The conditions $xy=e$ and $yx=f$ imply that $L_x$ and $L_y$ are mutually inverse so $L_y:eR\cong_{\gr} (\gamma)fR.$ 

The equivalence (4) $\Leftrightarrow$ (2) can be shown analogously. The condition (3) implies (4) since if $x,y$ are as in (3), then $exf$ and $fye$ are elements as in (4). The converse (4) $\Rightarrow$ (3) directly holds.   
\end{proof}

We use the following lemma in the proofs of Propositions \ref{graded_ur_of_end_ring} and \ref{graded_unit_regular}.

\begin{lemma} If $R$ is a $\Gamma$-graded ring, $A$ is a graded $R$-module, $S=\END_R(A),$ $\gamma\in\Gamma,$ and $e,f$ homogeneous idempotents in $S$, then the following conditions 
are equivalent.
\begin{enumerate}
\item $eS\cong_{\gr} (\gamma)fS$ \hskip2cm (2) $\,\,eS_\epsilon\cong \left( (\gamma)fS\right)_\epsilon= fS_\gamma$ \hskip2cm (3) $\,\,eA\cong_{\gr} (\gamma)fA$
\end{enumerate}
\label{endo_to_module}
\end{lemma}
\begin{proof}
The equality in condition (2) follows by definition. We show (1) $\Leftrightarrow$ (2) and (1) $\Leftrightarrow$ (3).

An isomorphism $\phi: eS\cong_{\gr} (\gamma)fS$ restricts to $eS_\epsilon\cong ((\gamma)fS)_\epsilon =fS_\gamma$
so (1) implies (2). Conversely, if $\phi_\epsilon: eS_\epsilon\cong (\gamma)fS_\epsilon,$ then $\phi,$ defined by $ex\mapsto \phi_{\epsilon}(e)x,$ is a graded isomorphism $eS\cong_{\gr} (\gamma)fS.$

If (1) holds, then $e=xy, f=yx$ for some $x\in eS_{\gamma^{-1}} f$ and $y\in fS_\gamma e$ by Lemma \ref{graded_equivalence}. So, $y$ restricted on $eA$ is a graded isomorphism $eA\cong_{\gr}(\gamma) fA$ which shows (3). Conversely, if (3) holds and $y$ is a graded isomorphism $eA\to (\gamma)fA$ with inverse $x,$ then $y$ can be extended to an element of $S_\gamma$ by $y((1-e)A)=0.$ Similarly, $x$ can be extended to an element of $S_{\gamma^{-1}}$ by $x((1-f)A)=0.$ Since $xy(a)=xye(a)=e(a),$ $yx=e$ and, similarly, $xy=f.$ Thus, (1) holds by Lemma \ref{graded_equivalence}.    
\end{proof}

We also use Lemma \ref{lemma_left_multiplication} below. Let $R$ be a $\Gamma$-graded ring, $x\in R_\gamma,$ and let $L_x$ denote the left multiplication by $x.$
Then $\ker L_x$ is a graded right ideal of $R$ and $L_x\in \Hom_R(R, R)_\gamma=\Hom_R(R, (\gamma)R)_\epsilon.$ So, $xR$ is a graded submodule of $(\gamma)R$ which implies that $(\gamma^{-1})xR$ is a graded right ideal of $R.$ Thus, the following two are short exact sequences of graded right $R$-modules. 
\[\xymatrix{0\ar[r]& \,\ker L_x\,\ar[r]&\; R\;\ar[r]^{L_x}&\; xR\;\ar[r]&0}\]
\[\xymatrix{0\ar[r]& (\gamma^{-1})xR\ar[r]& R\ar[r]& (\gamma^{-1})\coker L_x\ar[r]&0} \]

\begin{lemma}
If $R$ is a $\Gamma$-graded ring, $x\in R_{\gamma}$ for $\gamma\in\Gamma,$ $L_x$ is the left multiplication by $x,$ and $x=xyx$ for some homogeneous element $y,$ then $L_x$ is a graded isomorphism of $yxR,$ $xR=(\gamma)xyR,$ and $(1-yx)R\cong_{\gr}(\gamma)(1-xy)R$ if and only if $\ker L_x\cong_{\gr}\coker L_x.$ 
\label{lemma_left_multiplication} 
\end{lemma}
\begin{proof}
The relation $x=xyx$ implies that $y\in R_{\gamma^{-1}},$ that $\ker L_x=(1-yx)R$ and $(\gamma^{-1})xR=xyR,$ and that $L_x: yxR\to xyxR=xR$ is a graded isomorphism. So, $(\gamma^{-1})\coker L_x\cong_{\gr}(1-xy)R$ and thus $\coker L_x\cong_{\gr}(\gamma)(1-xy)R.$ 
\end{proof}

Lemma \ref{invertible_homogeneous_elements}, needed for Proposition \ref{proposition_matrices}, shows that UR$_{\gr}$ is a rather strong requirement. 
\begin{lemma}
If a graded ring $R$ is graded unit-regular, then every nonzero component contains an invertible element.  
\label{invertible_homogeneous_elements}
\end{lemma}
\begin{proof}
If $0\neq x\in R_\gamma$ and $R$ is graded unit-regular, then there is a homogeneous invertible element $u$ such that $x=xux.$ This last condition forces $u$ to be in $R_{\gamma^{-1}}$ and so its inverse is in $R_\gamma.$
\end{proof}

\section{Graded unit-regular rings and graded cancellability}\label{section_unit_regular_and_cancellability}

\subsection{Graded unit-regular rings}\label{subsection_unit_regular}  We recall \cite[Theorem 4.1]{Goodearl_book} stating that the following conditions are equivalent for a ring $R$, a right $R$-module $A$, and $S=\End_R(A).$  
\begin{enumerate}
\item $S$ is unit-regular. 
 
\item $S$ is regular and $A$ satisfies internal cancellation IC$(A).$  

\item $S$ is regular and $e\sim f$ implies $1-e\sim 1-f$ for all idempotents $e,f\in S.$
\end{enumerate}
Propositions \ref{graded_ur_of_end_ring} and \ref{graded_unit_regular} generalize these equivalences if $R$ is graded. Recall that IC$_{\gr}(A)$ denotes the condition $A=B\oplus C=D\oplus E$ and $B\cong_{\gr} D$ implies $C\cong_{\gr} E$ for graded modules $A,B,C,D,E.$

\begin{proposition} Let $R$ be a $\Gamma$-graded ring, $A$ a graded $R$-module, and $S_\epsilon$ be the component of $S=\END_R(A)$ corresponding to the identity $\epsilon\in \Gamma$. Then, the following conditions are equivalent. 
\begin{enumerate}
\item $S_\epsilon$ is unit-regular. 
 
\item $S_\epsilon$ is regular and $A$ satisfies graded internal cancellation IC$_{\gr}(A).$  

\item $S_\epsilon$ is regular and $e\sim f$ implies $1-e\sim 1-f$ for all idempotents $e,f\in S_\epsilon.$
\end{enumerate}
If  $A$ is finitely generated, the above statements hold for $S=\End_R(A).$
\label{graded_ur_of_end_ring}
\end{proposition} 
\begin{proof}
To show (1) $\Rightarrow$ (2), let $A=B\oplus C=D\oplus E$ and $x:B\cong_{\gr}D.$ Extend $x$ to an element of $S_\epsilon$ by mapping $C$ to 0. Let $u\in S_\epsilon$ be invertible and such that $xux=x$. Then, $(1-ux)A=\ker x=C$ and $uxA=uD$ so $u$ maps $D=xA$ onto $uxA$ and so $u$ maps $E$ onto $(1-ux)A=C.$ Hence $C\cong_{\gr} E.$ 

To show (2) $\Rightarrow$ (3), let $e,f\in S_\epsilon$ be idempotents such that $e\sim f$ so $eA\cong_{\gr} fA.$ By (2), $(1-e)A\cong_{\gr} (1-f)A$ which implies that $1-e\sim 1-f$ as elements of $S_\epsilon$ by Lemma \ref{endo_to_module}. 

To show (3) $\Rightarrow$ (1), let $x\in S_\epsilon$ and $y\in S_\epsilon$ be such that $xyx=x.$ Then $e=xy$ and $f=yx$ are idempotents of $S_\epsilon$ such that $e\sim f.$ By the assumption, $1-e\sim 1-f.$ So, there are $u\in (1-e)S_\epsilon(1-f), v\in (1-f)S_\epsilon(1-e)$ such that $uv=1-e$ and $vu=1-f.$ Since $x\in eS_\epsilon f$ and $yxy\in fS_\epsilon e,$ $yxy+v\in S_\epsilon$ is invertible with inverse $x+u$ and $x(yxy+v)x=x.$
\end{proof}

Recall the condition IC$_{\gr}$ from the introduction and let Mat$_\epsilon$ be the property below. 
\begin{itemize}
\item[IC$_{\gr}\;\,$:] IC$_{\gr}(P)$ holds for every finitely generated graded projective module $P.$
\item[Mat$_\epsilon$:] The $\epsilon$-component of $\M_n(R)(\gamma_1,\ldots,\gamma_n)$ is unit-regular for every $n$ and every $\gamma_1,\ldots,\gamma_n\in \Gamma.$
\end{itemize}

\begin{proposition}
Let $R$ be a $\Gamma$-graded ring. The following conditions are equivalent. 
\begin{enumerate}
\item Mat$_\epsilon$ holds for $R$. 
 
\item Mat$_\epsilon$ holds for $\M_m(R)(\delta_1,\ldots,\delta_m)$ for every positive integer $m$ and every $\delta_1,\ldots, \delta_m\in \Gamma.$ 
\end{enumerate}
These two conditions imply the condition IC$_{\gr}.$ If $\M_n(R)(\gamma_1,\ldots,\gamma_n)_\epsilon$ is regular for every $n$ and every $\gamma_1,\ldots,\gamma_n\in\Gamma$, then (1), (2) and IC$_{\gr}$ are equivalent.
\label{matricial_unit-regularity} 
\end{proposition}
\begin{proof} 
Since $(\gamma^{-1})(\delta^{-1})R(\delta)(\gamma)=((\gamma\delta)^{-1})R(\gamma\delta)$ for $\gamma,\delta\in\Gamma,$ we have that 
\[\M_n(\M_m(R)(\delta_1,\ldots,\delta_m))(\gamma_1,\ldots,\gamma_n)=\M_{nm}(R)(\gamma_1\delta_1, \ldots, \gamma_1\delta_m,\; \ldots\ldots \ldots,\; \gamma_n\delta_1,\ldots, \gamma_n\delta_m)\] for all positive integers $m$ and $n$ and all $\gamma_1,\ldots,\gamma_n, \delta_1,\ldots,\delta_m\in \Gamma.$  So, assuming (1) is sufficient for (2) and the converse trivially holds.  

Since IC$_{\gr}(\underline{\hskip.3cm})$ is preserved under formation of graded direct summands, IC$_{\gr}$ holds iff  IC$_{\gr}(F)$ holds for every finitely generated graded free module $F$. Every such module $F$ is of the form $\bigoplus_{i=1}^n(\gamma_i^{-1})R$ for some $n$ and some $\gamma_1,\ldots,\gamma_n.$ Since $\End_R(F)=\END_R(F)=\M_n(R)(\gamma_1,\ldots, \gamma_n),$
if $\End_R(F)_\epsilon$ is unit-regular then  IC$_{\gr}(F)$ holds by Proposition \ref{graded_ur_of_end_ring}. If $\End_R(F)_\epsilon$ is regular, then IC$_{\gr}(F)$ implies that $\End_R(F)_\epsilon$ is unit-regular also by Proposition \ref{graded_ur_of_end_ring}.    
\end{proof}

\begin{remark}
\label{matrices_over_reg_are_reg} 
Note that the assumption in the last sentence of Proposition \ref{matricial_unit-regularity} is automatically satisfied if $R$ is graded regular. Indeed, by the graded analogue of \cite[Theorem 1.7]{Goodearl_book}, graded regularity is passed to graded matrix algebras. The proof is analogous to the nongraded case: if $R$ is graded regular, then it is direct to check that $(\gamma^{-1})R(\gamma)$ is graded regular for every $\gamma\in \Gamma.$ So, $(\gamma_i^{-1})R(\gamma_i)\cong_{\gr}e_{ii}\M_n(R)(\gamma_1,\ldots, \gamma_n)e_{ii}$ is graded regular for all the standard matrix units $e_{ii}$ for any $n$ and $\gamma_1,\ldots, \gamma_n\in \Gamma.$ Then one shows that $\M_n(R)(\gamma_1,\ldots, \gamma_n)$ is graded regular by induction analogously to the proof in the nongraded case (see \cite[Lemma 1.6]{Goodearl_book}). 
This shows that if $R$ is graded regular, then $\M_n(R)(\gamma_1,\ldots, \gamma_n)$ is graded regular and, consequently, $\M_n(R)(\gamma_1,\ldots, \gamma_n)_\epsilon$ is regular.
\end{remark}

In Proposition \ref{graded_unit_regular}, we relate UR$^{\w}_{\gr}$ and UR$_{\gr}$ of $\END_R(A)$ for a graded module $A$ with the following weak and strong internal cancellation properties of $A$ respectively.  
\begin{enumerate}
\item[IC$^{\w}_{\gr}(A)$:]\hskip.5cm  $A=B\oplus C=D\oplus E$ and $B\cong_{\gr} (\gamma)D$  for some $\gamma\in\Gamma$ implies $C\cong E.$
\item[IC$^{\s}_{\gr}(A)$:]\hskip.5cm  $A=B\oplus C=D\oplus E$ and $B\cong_{\gr} (\gamma)D$  for some $\gamma\in\Gamma$ implies $C\cong_{\gr} (\gamma)E.$
\end{enumerate}

In Proposition \ref{graded_unit_regular}, we relate the properties UR$^{\w}_{\gr}$ and UR$_{\gr}$ with IC$^{\w}_{\gr}(A)$ and IC$^{\s}_{\gr}(A)$ respectively. 

\begin{proposition}
Let $R$ be a $\Gamma$-graded ring, $A$ be a graded right $R$-module, and $S=\END_R(A).$ The following conditions are equivalent. 
\begin{enumerate}
\item[(1$^{\w}$)] $S$ is weakly graded unit-regular. 
 
\item[(2$^{\w}$)] $S$ is graded regular and $A$ satisfies weak graded internal cancellation IC$^{\w}_{\gr}(A)$.  

\item[(3$^{\w}$)] $S$ is graded regular and  $eA\cong_{\gr} (\gamma)fA$ for some $\gamma\in \Gamma,$ implies $(1-e)A\cong (1-f)A$ for all homogeneous idempotents $e,f\in S.$
\end{enumerate}
The following conditions are also equivalent.
\begin{enumerate}
\item[(1$^{\s}$)] $S$ is graded unit-regular. 
 
\item[(2$^{\s}$)] $S$ is graded regular and $A$ satisfies strong graded internal cancellation IC$^{\s}_{\gr}(A)$.  

\item[(3$^{\s}$)] $S$ is graded regular and $eA\cong_{\gr} (\gamma)fA$ for some $\gamma\in \Gamma,$ implies $(1-e)A\cong (\gamma)(1-f)A$ for all homogeneous idempotents $e,f\in S.$
\end{enumerate}
If $A$ is finitely generated, then the above statements hold for  $S=\End_R(A).$
\label{graded_unit_regular}
\end{proposition}
\begin{proof}
Let us show (1$^{\w}$)$\Rightarrow$(2$^{\w}$) and (1$^{\s}$) $\Rightarrow$(2$^{\s}$).
Let $A=B\oplus C=D\oplus E$ and $x:B\cong_{\gr}(\gamma)D.$ Extend $x$ to $A$ by $xC=0.$ So, $x\in \HOM_R(A, (\gamma)A)_\epsilon=\END_R(A)_\gamma=S_\gamma.$ Assuming (1$^{\w}$), there is invertible $u\in S$ such that $x=xux.$ By the proof of (1)$\Rightarrow$(2) of Proposition \ref{graded_ur_of_end_ring}, we obtain $C\cong E.$ Assuming (1$^{\s}$), such $u$ can be found in $S_{\gamma^{-1}}.$  Then, $(1-ux)A=\ker x=C$ and $uxA=u(\gamma)D$ so $u$ maps $(\gamma)D=xA$ onto $uxA$ and so $u$ maps $(\gamma)E$ onto $(1-ux)A=C.$ Hence $C\cong_{\gr} (\gamma)E.$ 

Let us show (2$^{\w}$)$\Rightarrow$(3$^{\w}$) and (2$^{\s}$)$\Rightarrow$(3$^{\s}$). Assume that $eA\cong_{\gr} (\gamma)fA$ for some $\gamma\in \Gamma.$ Condition (2$^{\w}$) implies that $(1-e)A\cong(1-f)A$ and condition (2$^{\s}$) implies that $(1-e)A\cong_{\gr}(\gamma)(1-f)A.$ 

Let us show (3$^{\w}$)$\Rightarrow$(1$^{\w}$) and (3$^{\s}$)$\Rightarrow$(1$^{\s}$). Let $x\in S_\gamma.$ Under either (3$^{\w}$) or (3$^{\s}$), there is $y\in S_{\gamma^{-1}}$ such that $xyx=x.$ Then $e=xy$ and $f=yx$ are homogeneous idempotents and  $eA\cong_{\gr}(\gamma^{-1})fA$
by Lemmas \ref{graded_equivalence} and \ref{endo_to_module}. Condition (3$^{\w}$) implies that $(1-e)A\cong (1-f)A$ and condition (3$^{\s}$) that $(1-e)A\cong_{\gr} (\gamma^{-1})(1-f)A.$ In the second case, there are $u\in (1-e)S_\gamma(1-f), v\in (1-f)S_{\gamma^{-1}}(1-e)$ such that $uv=1-e$ and $vu=1-f$ by Lemmas \ref{graded_equivalence} and \ref{endo_to_module}. Then $yxy+v\in S_{\gamma^{-1}}$ is invertible with inverse $x+u\in S_\gamma$ and $x(yxy+v)x=x.$ In the first case, there are $u\in(1-e)S(1-f)$ and $v\in(1-f)S(1-e)$ such that  $uv=1-e$ and $vu=1-f$ and the rest of the prior arguments show that $yxy+v$ is invertible and that $x(yxy+v)x=x.$
\end{proof}

The implication UR$_{\gr}\Rightarrow$ UR$^{\w}_{\gr}$ is direct and it is strict by Example \ref{example_cancellability_conditions}.
It is also direct to see that IC$^{\s}_{\gr}(\underline{\hskip.3cm})$ implies both IC$^{\w}_{\gr}(\underline{\hskip.3cm})$ and IC$_{\gr}(\underline{\hskip.3cm}).$ Hence, UR$_{\gr}$ implies IC$_{\gr}(R).$ However, it is not direct to see that UR$_{\gr}\Rightarrow$ IC$_{\gr}.$  This implication follows from Theorem \ref{C_and_IC}.    

\subsection{Graded cancellability}
\label{subsection_graded_cancellable}

The cancellation property has a favorable feature that a finite direct sum is cancellable if and only if each of its terms is cancellable (see \cite[Proposition 3.3]{Lam_cancellation_properties}). Relating IC$_{\gr}(\underline{\hskip.3cm})$ with C$_{\gr}(\underline{\hskip.3cm})$ in Theorem \ref{C_and_IC}, we show that IC$_{\gr}(\underline{\hskip.3cm})$ is closed under the formation of direct sums of modules if the ring is graded regular. The conditions IC$_{\gr}(\underline{\hskip.3cm})$ and IC$^{\s}_{\gr}(\underline{\hskip.3cm})$ alone are not closed for finite direct sums (consider \cite[Example 3.2 (3)]{Lam_cancellation_properties} and grade the ring trivially). 

If $R$ is a $\Gamma$-graded ring and $A$ a graded module, recall that we say that C$_{\gr}(A)$ holds in a category of graded $R$-modules $\Mod$ if $A\oplus B\cong_{\gr} A\oplus C$ implies $B\cong_{\gr} C$
for all graded modules $B$ and $C$ in $\Mod.$ If $\Proj_{\gr}$ is the category of finitely generated graded projective modules, and $A$ in $\Proj_{\gr},$ we consider C$_{\gr}(A)$ only in $\Proj_{\gr}$ so we abbreviate ``C$_{\gr}(A)$ holds in $\Proj_{\gr}$'' as ``C$_{\gr}(A)$ holds''. \footnote{ 
One could also consider the weak and strong graded cancellability of a module $A\in \Proj_{\gr}$ analogously to the weak and strong graded internal cancellation as follows. 
\begin{enumerate}
\item[C$^{\w}_{\gr}(A)$:] \hskip.5cm $A\oplus B\cong_{\gr}  (\gamma)A\oplus C$ implies $\;\;\;\;\;B\cong\;\; C$  for every $\gamma\in\Gamma$ and every $B,C\in \Proj_{\gr}.$

\item[C$^{\s}_{\gr}(A)$:] \hskip.5cm $A\oplus B\cong_{\gr} (\gamma)A\oplus C$ implies $(\gamma)B\cong_{\gr}C$ for every $\gamma\in\Gamma$ and every $B,C\in \Proj_{\gr}.$
\end{enumerate}
It is direct to show that C$^{\s}_{\gr}(\underline{\hskip.3cm})$ $\Rightarrow$ C$_{\gr}(\underline{\hskip.3cm})$ and that C$^{\s}_{\gr}(\underline{\hskip.3cm})$  $\Rightarrow$ C$^{\w}_{\gr}(\underline{\hskip.3cm}).$ One can show that the conditions C$^{\w}_{\gr}(\underline{\hskip.3cm})$ and C$^{\s}_{\gr}(\underline{\hskip.3cm})$ do not share the nice addition and shift-invariant properties of C$_{\gr}(\underline{\hskip.3cm})$. } 

Note that C$_{\gr}(A)$ holds if and only if C$_{\gr}((\gamma)A)$ holds for any $\gamma\in\Gamma$ as it is directly to check. 
In addition, C$_{\gr}(A\oplus B)$ holds if and only if C$_{\gr}(A)$ and C$_{\gr}(B)$ hold. This can also easily be checked by the argument completely analogous to the nongraded case (see \cite[Proposition 3.3]{Lam_cancellation_properties}). \footnote{ All 
the statements made in this section so far are true if $\Proj_{\gr}$ is replaced by any category of graded modules.}
Thus,  
\begin{center} 
C$_{\gr}(R)$ holds \hskip.4cm if and only if \hskip.4cm C$_{\gr}(P)$ holds for any $P\in\Proj_{\gr}.$
\end{center}
Hence, C$_{\gr}(R)$ holds if and only if the $\Gamma$-monoid $\V^{\Gamma}(R)$ (see \cite[Section 1.3]{Lia_realization}) is cancellative.

In the nongraded case, C$(\underline{\hskip.3cm}) \Rightarrow$ IC$(\underline{\hskip.3cm})$ and the converse holds if $R$ is regular  (\cite[Theorem 4.5]{Goodearl_book}). We show the graded versions of these statements and relate C$_{\gr}$ with UR$_\epsilon$ and IC$_{\gr}$. 

\begin{theorem}
Let $R$ be a $\Gamma$-graded ring and $P\in\Proj_{\gr}$. 
\begin{enumerate}
\item If C$_{\gr}(P)$ holds, then  IC$_{\gr}(P)$ holds and the converse holds if $R$ is graded regular. 
  
\item If C$_{\gr}(R)$ holds, then IC$_{\gr}$ holds and the converse holds if $R$ is graded regular. 
 
\item If $R$ is graded regular, then $R_\epsilon$ is unit-regular if and only if C$_{\gr}(R)$ holds. Hence, the conditions 
UR$_\epsilon,$  C$_{\gr}(R),$  C$_{\gr},$  IC$_{\gr},$ and Mat$_\epsilon$ are all equivalent for a graded regular ring $R.$
\end{enumerate}
\label{C_and_IC}
\end{theorem}
\begin{proof}
Assuming that  C$_{\gr}(P)$ holds, let $P=A\oplus B=C\oplus D$ and $A\cong_{\gr} C.$ Then $A\oplus B\cong_{\gr}A\oplus D.$ Since C$_{\gr}(P)$ implies C$_{\gr}(A)$, we have that $B\cong_{\gr} D.$

Let $R$ be graded regular and let $P\oplus A\cong_{\gr} P\oplus B$ for some $A,B\in \Proj_{\gr}$ now. By \cite[Theorem 2.8]{Goodearl_book}, two direct sum decompositions of a finitely generated projective module over a regular ring have isomorphic refinements. The graded version of this statement can be shown by a proof completely analogous to the proof of \cite[Theorem 2.8]{Goodearl_book}. So, there are graded decompositions $P=P_1\oplus P_2$ and $A=A_1\oplus A_2$ such that $P_1\oplus A_1\cong_{\gr} P$ and $P_2\oplus A_2\cong_{\gr} B.$ Hence $P_1\oplus A_1\cong_{\gr} P=P_1\oplus P_2$ implies $A_1\cong_{\gr} P_2$ by IC$_{\gr}(P)$. Thus, $A=A_1\oplus A_2\cong_{\gr} P_2\oplus A_2\cong_{\gr}B.$

To show (2), assume that  C$_{\gr}(R)$ holds. Since C$_{\gr}(\underline{\hskip.3cm})$ is closed under taking finite direct sums and graded direct summands, C$_{\gr}(P)$ holds for any $P\in \Proj_{\gr}.$ By statement (1), IC$_{\gr}$ holds and the converse holds if $R$ is graded regular. 

To show (3), note that if $R$ is graded regular, then UR$_\epsilon$ and  IC$_{\gr}(R)$ are equivalent by Proposition \ref{graded_ur_of_end_ring}. By part (1),  IC$_{\gr}(R)$ and C$_{\gr}(R)$ are equivalent. By part (2), C$_{\gr}(R)$ and IC$_{\gr}$ are equivalent. The conditions C$_{\gr}(R)$ and C$_{\gr}$ are equivalent since C$_{\gr}(\underline{\hskip.3cm})$ is closed under taking finite direct sums and graded direct summands and the conditions  IC$_{\gr}$ and Mat$_\epsilon$ are equivalent by Proposition \ref{matricial_unit-regularity}.  
\end{proof}

If $R$ is a ring and $e$ and idempotent, the ring $eRe$ is called a corner. If $R$ is a graded ring and $e$ a homogeneous idempotent, the ring $eRe$ is a {\em graded corner.}  The property of being unit-regular, being directly finite and having stable range 1 are passed to corners. The proofs of these facts involve consideration of an element $x+1-e$ of $R$ for any element $x$ of $eRe$ (see \cite[Theorem, \S 2]{Lam_Murray} for unit-regularity, \cite[Theorem 2.8]{Vaserstein2} for stable range 1 and \cite[7.3]{Berberian_web} for direct finiteness). This is problematic for graded rings since if $x$ is in $R_\gamma$ for $\gamma\neq \epsilon$ and if $e\neq 1$, then $x+1-e$ is not homogeneous so none of the proofs of the nongraded cases can be adjusted to the graded cases. In Example \ref{example_corners}, we show that graded unit-regularity is {\em not} necessarily passed to graded corners. 
 
If a graded property $P_{\gr}$ is closed under formation of graded matrix algebras and graded corners, then it is graded Morita invariant (in the sense of \cite[Section 2.3]{Roozbeh_book}). While unit-regularity is Morita invariant, graded unit-regularity is not graded Morita invariant as we have seen in the introduction (also by Proposition \ref{proposition_matrices} and by Example \ref{example_corners}). However, Reg$_{\gr}$+UR$_\epsilon$ {\em is} graded Morita invariant by Corollary \ref{morita_invariant}. This is another advantage of Reg$_{\gr}$+UR$_\epsilon$ over UR$_{\gr}.$

\begin{corollary}
The property Reg$_{\gr}$+UR$_\epsilon$ is graded Morita invariant. 
\label{morita_invariant} 
\end{corollary}
\begin{proof}
The property Reg$_{\gr}$ is closed under formation of graded matrix algebras (see Remark \ref{matrices_over_reg_are_reg}) and graded corners (direct to check). By Theorem \ref{C_and_IC}, Reg$_{\gr}$ $\Rightarrow$ (UR$_\epsilon$ $\Leftrightarrow$ Mat$_\epsilon$), so UR$_\epsilon$ is closed under formation of graded matrix algebras. Since $(eRe)_\epsilon=eR_\epsilon e$ if $R$ is graded ring and $ee=e\in R_\epsilon,$ 
UR$_\epsilon$ is closed under formation of graded corners as UR is closed under formation of corners. 
\end{proof}

Propositions \ref{matricial_unit-regularity} and \ref{graded_unit_regular} and Theorem \ref{C_and_IC} show the diagram below. {\small
\begin{center}
\begin{tabular}{ccccc}
\begin{tabular}{|lll|}\hline
 Reg$_{\gr}$+UR$_\epsilon$& $\Leftrightarrow$&\\ Reg$_{\gr}$+C$_{\gr}(R)$& $\Leftrightarrow$ & Reg$_{\gr}$+IC$_{\gr}(R)$\\ \hline
\end{tabular}
&$\Leftarrow$&
\begin{tabular}{|l|}\hline
UR$_{\gr}$ $\;\;\;\Leftrightarrow$\\ Reg$_{\gr}$+IC$^{\s}_{\gr}(R)$\\ \hline
\end{tabular}
&$\Rightarrow$& 
\begin{tabular}{|l|}\hline 
UR$^{\w}_{\gr}$ $\;\;\;\Leftrightarrow$ \\ Reg$_{\gr}$+IC$^{\w}_{\gr}(R)$\\ \hline
\end{tabular} 
\end{tabular}
\end{center}}

We present examples showing that both implications above are strict and that UR$_\epsilon$ and UR$^{\w}_{\gr}$ are not equivalent even if a ring is graded regular, and that UR and UR$_{\gr}$ are independent. We also show that the following relations hold.   
\begin{enumerate}
\item Reg$_{\gr}$+UR$_\epsilon$ $\nRightarrow$ UR$_{\gr},$  UR $\nRightarrow$ UR$_{\gr},$ UR$^{\w}_{\gr}$ $\nRightarrow$ UR$_{\gr}.$
\item UR$^{\w}_{\gr}$ $\nRightarrow$ UR, UR$_{\gr}$ $\nRightarrow$ UR, Reg$_{\gr}$+UR$_\epsilon$ $\nRightarrow$ UR.
\item Reg$_{\gr}$+UR$_\epsilon$ $\nRightarrow$ C$(R)$.  
\item C$(R)$ $\nRightarrow$ UR$^{\w}_{\gr}$ and C$_{\gr}(R)$ $\nRightarrow$ UR$^{\w}_{\gr}.$ 
\end{enumerate}
\begin{example} In (1), (2) and (3) below, $K$ is any field trivially graded by $\Zset.$
\begin{enumerate}

\item The graded ring $R=\M_2(K)(0,1)$ is not graded unit-regular as we have seen in the introduction. Since $R_0=\left[
\begin{array}{cc}
K & 0\\
0 & K
\end{array}\right],$ $R_0$ is unit-regular. Graded regularity is passed to graded matrix algebras (see Remark \ref{matrices_over_reg_are_reg}) so $R$ is graded regular. The ring $R=\M_2(K)$ is unit-regular and hence $R$ is weakly graded unit-regular. 
 
\item Let $R=K[x,x^{-1}],$ $\Zset$-graded as in Section \ref{subsection_matrices}. Then $R$ is a graded field so it is graded unit-regular, hence weakly graded unit-regular also. Since $R_0=K,$ $R_0$ is unit-regular. However, $R$ is not unit-regular (consider $1+x$ for example). 
 
\item Let $R$ be the Leavitt algebra $L(1,2)$ i.e. the universal example of a $K$-algebra $R$ such that $R\oplus R\cong R.$
Clearly, $R$ is not cancellable. The algebra $R$ can be represented as a Leavitt path algebra of the graph  $\;\;\;\;\;\xymatrix{\ar@(lu,ld)\bullet\ar@(ru,rd) }\;\;\;\;\;$ and it is naturally graded by $\Zset$ (see Section \ref{subsection_LPAs}). Since every Leavitt path algebra is graded regular and graded cancellable (by \cite[Theorem 9]{Roozbeh_regular} and \cite[Corollary 5.8]{Ara_et_al_Steinberg}), $R$ is such too and hence Reg$_{\gr}$+UR$_\epsilon$ holds by Theorem \ref{C_and_IC}.   

\item Let $R=\Zset.$ Then $\V(R)=\Zset^+$ so $R$ is cancellable. Consider $R$ trivially graded by $\Zset.$ Then $\V^{\Zset}(R)=\Zset^+[x,x^{-1}]$ (\cite[Example 3.1.5]{Roozbeh_book} has more details) so $R$ is graded cancellable. The ring $R$ is not regular, so it is not unit-regular and, since it is trivially graded, UR$^{\w}_{\gr}$ fails.    
\end{enumerate}
\label{example_cancellability_conditions}
\end{example}

\section{Graded stable range 1 and graded direct finiteness}
\label{section_graded_sr_1_and_DF}

\subsection{Graded stable range 1}\label{subsection_graded_sr_1} A regular ring is unit-regular if and only if it has stable range 1. First, we review some related terminology and show the graded version of this statement. 

A sequence of elements $a_1,\ldots, a_n$ of a ring $R$ is said to be right unimodular if $a_1R+\ldots+a_nR=R.$ If $R$ is $\Gamma$-graded, a sequence of elements $a_1,\ldots, a_n$ with $\deg(a_i)=\gamma_i,$ $i=1,\ldots,n,$ is graded right unimodular if $(\gamma_1^{-1})a_1R+\ldots+(\gamma_n^{-1})a_nR=R.$ Note that this last condition is equivalent with $\sum_{i=1}^n a_ix_i=1$ for some $x_1,\ldots,x_n.$ However, by replacing $x_i$ with its $\gamma_i^{-1}$-component $y_i,$ we obtain {\em homogeneous} elements $y_1,\ldots,y_n$ such that  $\sum_{i=1}^n a_iy_i=1.$ 
 
If $R$ is nongraded, recall that a sequence of unimodular elements $a_1,\ldots, a_n$ of $R$ is reducible if there are elements $b_1,\ldots,b_{n-1}$ such that $(a_1+a_nb_1)R+\ldots+(a_{n-1}+a_nb_{n-1})R=R.$ As opposed to the conditions with weak and strong versions, there is just one level of graded reducibility since the following two conditions are equivalent for $n\geq 2$ and a graded unimodular sequence $a_1,\ldots, a_n$ of elements of $R$ with $\deg(a_i)=\gamma_i, i=1,\ldots, n.$ 
\begin{enumerate}
\item There are elements $b_1,\ldots,b_{n-1}$ such that $a_i+a_nb_i\in R_{\gamma_i}$ for $i=1,\ldots, n-1$ and $(\gamma_1^{-1})(a_1+a_nb_1)R+\ldots+(\gamma_{n-1}^{-1})(a_{n-1}+a_nb_{n-1})R=R.$

\item  There are {\em homogeneous} elements $b_1,\ldots,b_{n-1}$ such that $a_i+a_nb_i\in R_{\gamma_i}$ for $i=1,\ldots, n-1$ and $(\gamma_1^{-1})(a_1+a_nb_1)R+\ldots+(\gamma_{n-1}^{-1})(a_{n-1}+a_nb_{n-1})R=R.$
\end{enumerate}
The first condition implies the second if we replace the elements $b_i$ with their $\gamma_n^{-1}\gamma_i$-components and the converse clearly holds. If any of the above two conditions are satisfied, we say that the sequence $a_1,\ldots, a_n$ is {\em graded reducible}. The second definition was used in \cite[Section 1.8]{Roozbeh_book}. 

Recall that the right stable range (or rank) of $R$ is at most $n,$ written $\sr^r(R)\leq n,$ if any right unimodular sequence of more than $n$ elements is reducible. If the smallest such $n$ exists, $\sr^r(R)=n$. If the smallest such $n$ does not exist, $\sr^r(R)=\infty.$ 
The range function $\sr_{\gr}^r$ is defined analogously using graded reducibility instead of reducibility and the left-sided version $\sr_{\gr}^l$ is defined similarly.

In the nongraded case, $\sr^r(R)\leq n$ if and only if every right unimodular sequence of $n+1$ elements is reducible (originally in \cite{Vaserstein1}, see also \cite[Proposition 1.3]{Lam_cancellation_properties}). The proof of \cite[Proposition 1.3]{Lam_cancellation_properties} generalizes step-by-step to the graded case. So, $\sr_{\gr}^r(R)\leq n$ if and only if every graded right unimodular sequence of $n+1$ elements is graded reducible. One can also show that $\sr^r(R)=n$ iff $\sr^l(R)=n$ (see \cite{Vaserstein1}), so one can denote $\sr^l$ and $\sr^r$ with $\sr$ only. We use the graded version of this result only in the case $n=1$ and include a proof for completeness.

\begin{lemma}
If $R$ is a $\Gamma$-graded ring, then $\sr_{\gr}^r(R)=1$ if and only if $\sr_{\gr}^l(R)=1.$
\end{lemma}
\begin{proof}
We adapt the proof of \cite[Theorem 1.8]{Lam_cancellation_properties} to the graded case. Let $\sr_{\gr}^r(R)=1$ and let $b\in R_\gamma$ and $d\in R_\delta$ be such that $Rb(\gamma^{-1})+Rd(\delta^{-1})=R.$ Thus, $ab+cd=1$ for some $a\in R_{\gamma^{-1}}$ and $c\in R_{\delta^{-1}}$ and so $(\gamma)aR+cdR=R.$ Hence, there is $x\in R_{\gamma^{-1}}$ such that $u=a+cdx\in R_{\gamma^{-1}}$ is right invertible. By \cite[Section 1.8]{Roozbeh_book}, if $\sr_{\gr}^r(R)=1,$ then a homogeneous element with a right inverse is invertible. Thus, $u$ is invertible. Let $v\in R_\gamma$ be its inverse. 
If $w=a+x(1-ba)$ and  $y=(1-bx)v,$ then $w\in R_{\gamma^{-1}}$ and $y\in R_{\gamma}.$ One checks that $w(1-bx)=(1-xb)u$ and $w(b+ycd)=1$ (for more details see \cite[Theorem 1.8]{Lam_cancellation_properties}). As $y\in R_\gamma,$ $b+ycd$ is in $R_\gamma$ also. Since $w(b+ycd)=1,$ $R(b+ycd)(\gamma^{-1})=R.$   
\end{proof}

This lemma allows us to shorten $\sr_{\gr}^r(R)=1$ and $\sr_{\gr}^l(R)=1$ to $\sr_{\gr}(R)=1$ and we say that $R$ has {\em graded stable range 1} in this case. The next proposition, stated without proof in \cite[Example 1.8.8]{Roozbeh_book}, relates this condition with graded unit-regularity. 

\begin{proposition}
If $R$ is a $\Gamma$-graded ring then $R$ is graded unit-regular if and only if $R$ is graded regular and $\sr_{\gr}(R)=1.$  
\label{graded_sr_and_gr_ur} 
\end{proposition}
\begin{proof}
Assume that $R$ is graded unit-regular and that $(\gamma^{-1})aR+(\delta^{-1})bR=R$ for some $a\in R_\gamma, b\in R_\delta.$ 
Let $a=aua$ and $b=bvb$ for some $u\in R_{\gamma^{-1}}$ invertible and $v\in R_{\delta^{-1}}.$ Then $au$ and $bv$ are idempotents in $R_\epsilon$ such that $(\gamma^{-1})aR=auR,$ $(\delta^{-1})bR=bvR$ (see Lemma \ref{lemma_left_multiplication}) so $auR+bvR=R.$
Since $bvR/(auR\cap bvR)\cong_{\gr} R/auR\cong_{\gr}(1-au)R,$ $auR\cap bvR$ is a graded summand of $bvR.$ Let $e\in R_\epsilon$ be an idempotent such that $bvR=(auR\cap bvR)\oplus eR.$ Then $R=auR\oplus eR.$  
If $L_{u^{-1}}$ is the left multiplication by $u^{-1},$ then $L_{u^{-1}}$ restricted on $uaR$ is $L_a: uaR\cong_{gr} aR=(\gamma)auR.$ On $(1-ua)R,$ $L_{u^{-1}}$ is $(1-ua)R\cong_{\gr} (\gamma)(1-au)R$ since $u^{-1}(1-ua)R=(1-au)u^{-1}R=(1-au)R.$ So, $u^{-1}=L_{u^{-1}}(1)=L_{u^{-1}}(ua+1-ua)=a+u^{-1}(1-ua)=a+(1-au)x$ for some $x\in R_\gamma.$ Since $(1-au)x\in eR\subseteq bvR,$ $(1-au)x=bvy$ holds for some $y\in R_\gamma.$ So, $a+bvy=u^{-1}\in R_\gamma$ is invertible.  

Conversely, assume that $\sr_{\gr}(R)=1$ and that $R$ is graded regular. If $a$ is in $R_\gamma,$ then $a=aba$ for some $b\in R_{\gamma^{-1}}$ and so $ab\in R_\epsilon$ is an idempotent. Since $1=ab+1-ab$ and $abR=(\gamma^{-1})aR,$ $R=(\gamma^{-1})aR+(1-ab)R.$ By the assumption that $\sr_{\gr}(R)=1,$ there is $y\in R_\gamma$ such that $a+(1-ab)y$ is invertible. If $u$ denotes its inverse, then $a=aba=ab(a+(1-ab)y)ua=abaua=aua.$  
\end{proof}

In \cite[Corollary 1.8.5]{Roozbeh_book}, it is shown that if $\Gamma$ is abelian and $R$ a graded ring with $\sr_{\gr}(R)=1,$ then $R$ is graded cancellable. In the proof, the relation $\End_R((\gamma)R)\cong_{\gr}R$ has been used. However,  if $\Gamma$ is nonabelian,  $\End_R((\gamma)R)\cong_{\gr}(\gamma)R(\gamma^{-1})$  may not be graded isomorphic to $R.$ For example, let $\Gamma=D_3=\langle a,b|a^3=b^2=1, ba=a^2b\rangle,$ $\Delta=\{1,b\}$ and let $R=K[\Delta]$ be $\Gamma$-graded by $R_\gamma=K\gamma$ if $\gamma\in \Delta$ and $R_\gamma=0$ otherwise. Then $R_b=Kb$ and $((a)R(a^{-1}))_b=R_{aba^{-1}}=R_{a^2b}=0$ so $R$ and $(a)R(a^{-1})$ are not graded isomorphic. 

The proof of \cite[Corollary 1.8.5]{Roozbeh_book} can still be modified to hold for nonabelian $\Gamma$ if the proof of the lemma below is used instead of the possibly false relation $\End_R((\gamma)R)\cong_{\gr}R.$  

\begin{lemma}
If $R$ is a $\Gamma$-graded ring and $\sr_{\gr}(R)=1,$ then $\sr_{\gr}(\End_R((\gamma)R))=1$ for every $\gamma\in \Gamma.$ 
\label{lemma_sr_M1}
\end{lemma}
\begin{proof}
Since $\End_R((\gamma)R)\cong_{\gr} (\gamma)R(\gamma^{-1}),$ we show that  $\sr_{\gr}((\gamma)R(\gamma^{-1}))=1.$ Let $a$ and $b$ be homogeneous elements of $(\gamma)R(\gamma^{-1})$ (and hence of $R$ as well) such that $ac+bd=1$ for some homogeneous $c,d\in (\gamma)R(\gamma^{-1}).$ So, $a,b,c,d$ are homogeneous elements of $R$ such that $ac+bd=1.$ By the assumption that  $\sr_{\gr}(R)=1,$ there is a homogeneous element $y$ such that $a+by$ is homogeneous and invertible. However, this also implies that $y$ and $a+by$ are homogeneous as elements of $(\gamma)R(\gamma^{-1})$ and that  $a+by$ is invertible as an element of  $(\gamma)R(\gamma^{-1}).$ 
\end{proof}

As a direct corollary of the lemma,  \cite[Corollary 1.8.5]{Roozbeh_book} holds even if $\Gamma$ is not abelian. In Corollary \ref{sr1_implies_cancellable}, we improve this statement by showing that the conclusion holds if the assumption $\sr_{\gr}(R)=1$ is replaced by the weaker condition $\sr(R_\epsilon)=1.$ This shows that the conclusion of \cite[Corollary 1.8.5]{Roozbeh_book} also holds under this weaker assumption and without assuming that $\Gamma$ is abelian. 

The implication  $\sr_{\gr}(R)=1 \Rightarrow$ $\sr_{\gr}(\M_n(R)(\gamma))=1$ for $n=1$ any $\gamma\in \Gamma$ shown in Lemma \ref{lemma_sr_M1} does not hold for $n>1.$ Indeed, if $R$ is $\M_2(K)(0,1)$ for a trivially $\Zset$-graded field $K$, then $\sr_{\gr}(K)=1$ and $R$ is a graded regular ring which is not graded unit-regular so $\sr_{\gr}(R)>1.$ This property of $\sr_{\gr}$ differs from the well-known property of $\sr$ that  $\sr(R)=1\Rightarrow \sr(\M_n(R))=1.$ Thus, 
\begin{center}
$\sr(R)=1\Rightarrow \sr(\M_n(R))=1\;\;\;\;$ and $\;\;\;\;\sr_{\gr}(R)=1\nRightarrow \sr_{\gr}(\M_n(R)(\gamma_1, \ldots,\gamma_n))=1.$ 
\end{center}

\subsection{Substitution}
\label{subsection_substitution}

A module has substitution if and only if its endomorphism ring has stable range 1. We show the graded version of this statement in Theorem \ref{substitution_and_sr}. This enables us to weaken the conditions of the Graded Cancellation Theorem  (\cite[Theorem 1.8.4]{Roozbeh_book}) and \cite[Corollary 1.8.5]{Roozbeh_book}.   

\begin{theorem} Let $R$ be a $\Gamma$-graded ring and $A$ a graded $R$-module. Then $\sr(\END_R(A)_\epsilon)\;=1$ if and only if $A$ has graded substitution. 
\label{substitution_and_sr} 
\end{theorem}
\begin{proof}
We adapt the proof of the nongraded case (see, for example, \cite[Theorem 4.4]{Lam_cancellation_properties}). Assume that $\sr(\END_R(A)_\epsilon)=1$ first, and let $A\oplus B=A'\oplus B'=M$ for some graded modules $M, A', B, B'$ such that $A\cong_{\gr}A'.$
Let $\phi$ and $\psi$ denote the graded isomorphism $A\to A'$ and its inverse, $\pi$ denote the natural graded projection $A\oplus B$ onto $A$ and $\iota$ denote the natural graded injection $A\to A\oplus B.$ Let $(f,g)$ denote the projection $\pi$ with respect to the decomposition $A'\oplus B'$ so that $\pi(a', b')=f(a')+g(b'),$ and let
$\left(\begin{array}{c}f'\\ g'\end{array}\right)$ denote the injection $\iota$ with respect to the decomposition $A'\oplus B'$ so that $\iota(a)=(f'(a), g'(a)).$ The relation $\pi\iota=1_A$ implies that $f\phi\psi f'+gg'=ff'+gg'=1_A.$ By the assumption $\sr(\END_R(A)_\epsilon)=1,$ there are $h,u\in\END_R(A)_\epsilon$ such that $u$ is invertible and $f\phi+gg'h=u.$ Let $C=\{(\phi(a), g'h(a))\in A'\oplus B'$ $|$ $a\in A\}=$Im $\left(\begin{array}{c}\phi\\ g'h \end{array}\right).$ 
Then $C$ is a graded submodule of $A'\oplus B'$ such that $(a',b')=(\phi\psi(a'), g'h\psi(a'))+(0, b'-g'h\psi(a'))\in C
\oplus B'$ for every $(a',b')\in A'\oplus B'.$ On the other hand, $C\oplus B=A'\oplus B'$ also since  $B=\ker (f,g)$ and $C=\{(a', b') | b'=g'h\psi(a')\}$ so that $(a', b')\in C$ implies that $0=f(a')+g(b')=f(a')+gg'h\psi(a')=u\psi(a')$ iff $a'=0.$  

Conversely, if the relation $ff'+gg'=1_A$ holds in $\END_R(A)_\epsilon,$ then $\pi=(f,g): A\oplus A\to A$ and $\iota=\left(\begin{array}{c}f'\\ g'\end{array}\right): A\to A\oplus A$ are graded homomorphisms such that $\pi\iota=1_A$ so that $A\oplus A$ splits as $\ker\pi\oplus$ Im $\pi.$ Since Im $\pi=A$ and $A$ has graded substitution, there is a graded module $C$ such that $A\oplus C=\ker \pi \oplus C.$ Let $\phi$ be any graded isomorphism of $A$ and $C.$ View $\phi$ as a map $A\to C\subseteq C\oplus A$ and represent it by $\left(\begin{array}{c}f_1\\ g_1\end{array}\right)$ for some graded maps $f_1: A\to C, g_1: A\to \ker\pi.$ Since $C$ is a complement of $A$, $f_1$ is invertible. Since $C$ is a complement of $\ker\pi,$ $\pi\phi$ is invertible. By construction, $\pi\phi=ff_1+gg_1$ and so $\pi\phi f_1^{-1}=f+gg_1f_1^{-1}.$ Hence,  if $h=g_1f_1^{-1},$ then $f+gh$ is an invertible element of $\END_R(A)_\epsilon.$    
\end{proof} 

The Graded Cancellation Theorem (\cite[Theorem 1.8.4]{Roozbeh_book}) states that $\sr_{\gr}(\End_R(A))=1$ implies C$_{\gr}(A)$ if $\Gamma$ is abelian and $A$ finitely generated. Since S$_{\gr}(A)\Rightarrow$ C$_{\gr}(A)$,
Theorem \ref{substitution_and_sr} shows that it is not necessary to require that $\Gamma$ is abelian and if $A$ is not finitely generated, $\END_R(A)$ can be considered instead of $\End_R(A)$. Theorem \ref{substitution_and_sr} also shows that the conclusion of \cite[Theorem 1.8.4]{Roozbeh_book} holds if the assumption $\sr_{\gr}(\END_R(A))=1$ is replaced by the weaker condition $\sr(\END_R(A)_\epsilon)=1.$    

Taking $R$ for $A$ in Theorem \ref{substitution_and_sr}, we have that $\sr(R_\epsilon)\;=1$ if and only if $R$ has graded substitution. Thus, Theorem \ref{substitution_and_sr} has the following corollary showing that \cite[Corollary 1.8.5]{Roozbeh_book} holds if the assumption $\sr_{\gr}(R)=1$ is replaced by the weaker condition $\sr(R_\epsilon)=1.$

\begin{corollary}
If $R$ is a $\Gamma$-graded ring with $\sr(R_\epsilon)=1,$ then $R$ is graded cancellable.  
\label{sr1_implies_cancellable}
\end{corollary}

\subsection{Graded directly finite rings}\label{subsection_graded_finiteness}

Recall that an $R$-module $A$ is directly finite (or Dedekind finite) if $A\oplus B\cong A$ implies $B=0$ for any module $B.$ In this case, we say that DF($A$) holds. If DF denotes the ring property $(\forall x)(\forall y)(xy=1\Rightarrow yx=1),$ then DF($A$) holds if and only if DF holds for $\End_R(A)$ (see \cite[Lemma 5.1]{Goodearl_book}). A ring $R$ is said to be directly finite if $R$ is a directly finite left (equivalently right) $R$-module and this requirement holds if and only if DF holds for $R.$ 

If $R$ is a graded ring and $A$ a graded $R$-module, consider the graded versions of DF$(A)$ and DF. 
\begin{enumerate}
\item[DF$_{\gr}(A)$:] \hskip.5cm $A\oplus B\cong_{\gr} A$ implies $B=0$  for any graded module $B.$

\item[DF$_{\gr}$:] \hskip.5cm $(\forall x\in H)(\forall y\in H)(xy=1\Rightarrow yx=1)$
\end{enumerate}
Then $\END_R(A)_\epsilon$ has DF iff DF$_{\gr}(A)$ holds. Indeed, if $\END_R(A)_\epsilon$ has DF, $A=B\oplus C,$ and $y: A\to C$ is a graded isomorphism, then $y^{-1}$ can be extended to an element $x$ of $\END_R(A)_\epsilon$ by mapping $B$ identically to zero. Since $xy=1_A,$ $yx$ is equal to $1_A$ by the assumption and so $b=yx(b)=y(0)=0$ for all $b\in B.$ Conversely, if DF$_{\gr}(A)$ holds and $x,y\in \END_R(A)_\epsilon$ are such that $xy=1_A,$ then $yA=yxA$ and $y$ is a graded isomorphism of $A=xyA$ and $yxyA=yA.$ Thus, $A=yA\oplus (1_A-yx)A$ implies that  $(1_A-yx)A=0$ by DF$_{\gr}(A),$ and so $yx=1_A.$

By \cite[Proposition 3.2]{Roozbeh_Ranga_Ashish}, DF$_{\gr}$ holds for $\END_R(A)$ if and only if the condition below holds. 
\begin{enumerate}
\item[DF$^{\s}_{\gr}(A)$:] \hskip.5cm $A\oplus B\cong_{\gr} (\gamma)A$ for some $\gamma\in\Gamma$ implies $B=0$ for any graded module $B.$
\end{enumerate}
If DF$^{\s}_{\gr}(A)$ holds, the authors of \cite{Roozbeh_Ranga_Ashish} say that $A$ is {\em graded directly finite}. 
A graded ring $R$ satisfies DF$_{\gr}$ if and only if any of the equivalent conditions DF$^{\s}_{\gr}(R_R)$ and  DF$^{\s}_{\gr}({}_RR)$ holds and $R_\epsilon$ satisfies DF if and only if any of the equivalent conditions DF$_{\gr}(R_R)$ and  DF$_{\gr}({}_RR)$ holds. By \cite[Example 3.3]{Roozbeh_Ranga_Ashish}, DF$^{\s}_{\gr}(A)$ is strictly stronger than DF$_{\gr}(A).$ 

The condition IC$(A)$ clearly implies DF$(A)$ and, by the same argument, we have the following.{\small      
\begin{center}
\begin{tabular}{ccc}
IC$^{\s}_{\gr}(\underline{\hskip.3cm})$ & $\Longrightarrow$ & DF$^{\s}_{\gr}(\underline{\hskip.3cm})$\\
$\Downarrow$ && $\Downarrow$\\
IC$_{\gr}(\underline{\hskip.3cm})$ & $\Longrightarrow$ & DF$_{\gr}(\underline{\hskip.3cm})$\\
\end{tabular}
\end{center}}
We also note that 
IC$^{\w}_{\gr}(\underline{\hskip.3cm})$ is sufficient to imply DF$^{\s}_{\gr}(\underline{\hskip.3cm}).$ Indeed, if IC$^{\w}_{\gr}(A)$ holds for a graded module $A$ and if $A\oplus B\cong_{\gr} (\gamma)A,$ then $(\gamma)(B)\cong 0.$ So, $(\gamma)B=0$ and hence $B=0.$ 

In contrast to the rows of the above diagram, we have that IC$_{\gr}\nRightarrow$ DF$_{\gr}.$
For example, consider the algebra $R$ from part (3) of Example \ref{example_cancellability_conditions}. By this example, 
Reg$_{\gr}+$UR$_\epsilon$ holds and so IC$_{\gr}$ holds by Theorem \ref{C_and_IC}. However, $R$ is graded isomorphic to a Leavitt path algebra of a graph which has a cycle with an exit and so DF$_{\gr}$ fails by \cite[Theorem 3.7]{Roozbeh_Ranga_Ashish}.  

We use the following proposition in the proofs of Theorem \ref{finite_graphs_ur_characterization} and Proposition \ref{LPA_characterizations}.

\begin{proposition}
If $R$ is a $\Gamma$-graded ring with $\sr_{\gr}(R)=1,$ then $R$ is graded directly finite. 
\label{graded_sr_implies_df} 
\end{proposition}
\begin{proof}
The proof is the graded version of the proof of \cite[Lemma 1.7]{Lam_cancellation_properties}.  
Let $x\in R_\gamma,y\in R_{\gamma^{-1}}$ be such that $xy=1$ and let $e=1-yx.$ Then $(\gamma)yR=yxR$ and so $R=(\gamma)yR+eR.$ By $\sr_{\gr}(R)=1,$ there is $z\in R_\gamma$ such that $y+ez$ is invertible. Since $xe=0$, $x(y+ez)=xy=1$ which implies that $x=(y+ez)^{-1}$ is invertible. So, $xy=1$ implies that $y$ is the inverse of $x$ and that $yx=1.$     
\end{proof}

\subsection{Cancellation properties of strongly graded rings} \label{subsection_strongly_graded}
If $R$ is strongly graded ($R_\gamma R_\delta=R_{\gamma\delta}$ for every $\gamma,\delta\in \Gamma$), 
the category of graded $R$-modules and the category of $R_\epsilon$-modules are equivalent by $A\cong_{\gr} A_\epsilon\otimes_{R_\epsilon} R\mapsto A_\epsilon$ with the inverse $B\cong (B\otimes_{R_\epsilon}R)_\epsilon\mapsto B\otimes_{R_\epsilon} R$ (\cite[Theorem 1.5.1]{Roozbeh_book}).  

\begin{proposition}
Let $R$ be a strongly $\Gamma$-graded ring and $A$ a graded $R$-module. The following hold. 
\begin{enumerate}
\item $A$ is graded internally cancellable if and only if $A_\epsilon$ is internally cancellable. 

\item $A$ is graded cancellable if and only if  $A_\epsilon$ is cancellable. 

\item $A$ has graded substitution if and only if $A_\epsilon$ has substitution. 

\item DF$_{\gr}(A)$ holds if and only if $A_\epsilon$ is directly finite. 
\end{enumerate}
\label{strongly_graded_cancellability} 
\end{proposition}

\begin{proof}
All four statements are shown similarly, using the equivalence of categories. We provide more details for the first condition and note that the proofs of (2), (3), and (4) are similar.

Assuming that IC$_{\gr}(A)$ holds, IC$_{\gr}(A_\epsilon\otimes_{R_\epsilon} R)$ also holds since $A$ and $A_\epsilon\otimes_{R_\epsilon} R$ are graded isomorphic. If $A_\epsilon=B_\epsilon\oplus C_\epsilon=D_\epsilon\oplus E_\epsilon$ for some right $R_\epsilon$-modules $B_\epsilon, C_\epsilon, D_\epsilon,$ and $E_\epsilon$ with $f: B_\epsilon\cong D_\epsilon,$ then
$A_\epsilon\otimes_{R_\epsilon} R=B_\epsilon \otimes_{R_\epsilon} R\oplus C_\epsilon \otimes_{R_\epsilon} R=D_\epsilon \otimes_{R_\epsilon} R\oplus E_\epsilon \otimes_{R_\epsilon} R$ and $f$ can be extended to the graded isomorphism $B_\epsilon \otimes_{R_\epsilon} R\to D_\epsilon \otimes_{R_\epsilon} R$ by $a\otimes r\mapsto f(a)\otimes r.$ By IC$_{\gr}(A),$ $C_\epsilon \otimes_{R_\epsilon} R\cong_{\gr} E_\epsilon \otimes_{R_\epsilon} R.$ Considering the $\epsilon$-components, we obtain that $C_\epsilon\cong (C_\epsilon \otimes_{R_\epsilon} R)_\epsilon \cong (E_\epsilon \otimes_{R_\epsilon} R)_\epsilon \cong E_\epsilon.$  

Assume that IC$(A_\epsilon)$ holds and that $A=B\oplus C=D\oplus E$ for some graded $R$-modules $B,C,D, E$ with $B\cong_{\gr}D.$ Then $A_\epsilon =B_\epsilon \oplus C_\epsilon =D_\epsilon \oplus E_\epsilon $ and $B_\epsilon \cong D_\epsilon.$ By IC$(A_\epsilon),$ $f:C_\epsilon \cong E_\epsilon$ for some $f.$ Such $f$ induces $\ol f:C_\epsilon \otimes_{R_\epsilon} R\cong_{\gr} E_\epsilon \otimes_{R_\epsilon} R$ so that $C\cong_{\gr}C_\epsilon \otimes_{R_\epsilon} R\cong_{\gr} E_\epsilon \otimes_{R_\epsilon} R\cong_{\gr} E.$
\end{proof}

The implications UR$_{\gr}\Rightarrow$ UR$_\epsilon,$ $\sr_{\gr}(R)=1 \Rightarrow \sr(R_\epsilon)=1$ and DF$_{\gr}\Rightarrow$ (DF holds for $R_{\epsilon}$) are strict even for strongly graded rings. Indeed, if $R$ is the ring from part (3) of Example \ref{example_cancellability_conditions}, then $R$ is strongly graded by \cite[Theorem 1.6.13]{Roozbeh_book}. By Example \ref{example_cancellability_conditions}, $R$ is graded regular and $R_\epsilon$ is unit-regular. Thus, $\sr(R_\epsilon)=1$ and $R_\epsilon$  is directly finite. However, DF$_{\gr}$ fails for $R$ as we noted in Section \ref{subsection_graded_finiteness} (by \cite[Theorem 3.7]{Roozbeh_Ranga_Ashish}). Hence,  $\sr_{\gr}(R)>1$ by Proposition \ref{graded_sr_implies_df} and UR$_{\gr}$ fails for $R$ by Proposition \ref{graded_sr_and_gr_ur}.

\subsection{Summary of relations}\label{subsection_summary}

The graded {\em module} properties we considered are related as follows.   
\begin{center}
S$_{\gr}(\underline{\hskip.3cm})\;\;\;$ $\Longrightarrow\;\;\;$ C$_{\gr}(\underline{\hskip.3cm})\;\;\;$  $\Longrightarrow\;\;\;$  IC$_{\gr}(\underline{\hskip.3cm})$  $\Longrightarrow\;\;\;$  DF$_{\gr}(\underline{\hskip.3cm})$
\end{center}
Note that these relations match the relations of the nongraded analogues in the diagram in \cite[Formula (4.2)]{Lam_cancellation_properties}. Considering rings from \cite[Examples 3.2(3) and 4.7]{Lam_cancellation_properties} and \cite[Example 5.10]{Goodearl_book} and grading them trivially by any group shows that the implications are strict. 

We also note that the diagram in Section \ref{subsection_graded_finiteness} illustrates that our use of $\s$ in the superscript of the module properties is consistent: the absence of $\s$ indicates that the property is obtained only by replacing ``module'' by ``graded module'' and ``homomorphism'' by ``graded homomorphism'' without considering the graded module shifts. So, for any graded module $A,$  
graded module-theoretic properties of $A$ correspond to properties of $\END_R(A)_\epsilon$ and strong graded module-theoretic properties of $A$ correspond to graded properties of $\END_R(A).$ 

The graded {\em ring} properties are related as follows by Proposition \ref{graded_sr_and_gr_ur} and Corollary \ref{sr1_implies_cancellable}. 
\begin{center}
UR$_{\gr}\;\;\;$ $\Longrightarrow$ $\;\;\;\sr_{\gr}(R)=1\;\;\;$ $\Longrightarrow\;\;\;$ $\sr(R_\epsilon)=1\;\;\;$ $\Longrightarrow\;\;\;$ C$_{\gr}(R)$ 
\end{center}

The implications are also strict. To see that the first implication is strict, consider any ring which has stable range 1 and which is not unit-regular (e.g. the ring $K[[x]]$ of power series of one variable over any field $K$, see \cite[Examples 1.6]{Lam_cancellation_properties}) and grade it trivially by any group. To see that the third implication is strict, consider the ring $R=\Zset$ trivially graded by $\Zset.$ Then $\sr(R_0)=\sr(\Zset)=2>1$ and C$_{\gr}(R)$ holds by part (4) of Example \ref{example_cancellability_conditions}.
To see that the second implication is strict, consider any graded regular ring $R$ which is not graded unit-regular and such that $R_\epsilon$ is unit-regular (e.g. the ring from part (1) of Example \ref{example_cancellability_conditions}). This example also shows that the middle implication in the diagram below, which holds by Proposition \ref{graded_sr_and_gr_ur} and Theorem \ref{C_and_IC}, is strict.     
\begin{center}
Reg$_{\gr}\;\;\;$ $\Longrightarrow\;\;\;$ $\left(\right.$ UR$_{\gr}\;\;$ $\Longleftrightarrow$ $\;\;\sr_{\gr}(R)=1\;\;$ $\Longrightarrow\;\;$ $\sr(R_\epsilon)=1\;\;$ $\Longleftrightarrow\;\;$ C$_{\gr}(R)$ $\left.\right)$ 
\end{center}

\section{Graded unit-regularity of some matrix and Leavitt path algebras}\label{section_matrices_and_LPAs}

In this section, $K$ is a field, trivially graded by the group of integers $\Zset,$ and $K[x^m, x^{-m}]$ is the ring of Laurent polynomials naturally $\Zset$-graded by $K[x^m, x^{-m}]_{mk}=Kx^{mk}$ and $K[x^m, x^{-m}]_{n}=0$ if $m$ does not divide $n.$ Note that this makes $K[x^m, x^{-m}]$ into a graded field. 

\subsection{Graded unit-regularity \texorpdfstring{$\Zset$}{TEXT}-graded matrix algebras over \texorpdfstring{$K$}{TEXT} and \texorpdfstring{$K[x^m, x^{-m}]$}{TEXT} }\label{subsection_matrices}

\begin{proposition} Let $m$ and $n$ be positive and $\gamma_1,\gamma_2,\ldots,\gamma_n$ arbitrary integers.  
\begin{enumerate}
\item $\M_n(K)(\gamma_1, \gamma_2\ldots,\gamma_n)$ is graded unit-regular if and only if $n=1$ or $\gamma_1=\gamma_2=\ldots=\gamma_n.$  

\item If the list $\gamma_1,\ldots,\gamma_n$ is such that all of $0,1,\ldots, m-1$ appear on it when it is considered modulo $m$, then $\M_n(K[x^m, x^{-m}])(\gamma_1,\ldots,\gamma_n)$ is graded unit-regular if and only if $n=km$ for some positive integer $k$ and the list $\gamma_1,\ldots,\gamma_n,$ considered modulo $m,$ is such that each $i=0,1,\ldots, m-1$ appears exactly $k$ times. 
\end{enumerate}
\label{proposition_matrices} 
\end{proposition}
\begin{proof} (1) We prove direction $\Rightarrow$ by contrapositive. Assume that $n>1$ and that not all $\gamma_1, \gamma_2\ldots,\gamma_n$ are equal. If $\gamma_i$ is the smallest of $\gamma_1\ldots, \gamma_n,$ then $\delta_1=\gamma_1-\gamma_i,\ldots, \delta_n=\gamma_n-\gamma_i$ is a list of nonnegative integers such that at least one is positive by the assumption that not all $\gamma_1,\ldots,\gamma_n$ are equal and at least one is zero by construction. By permuting the entries, we can assume that $\delta_1$ is zero and $\delta_2$ is positive. Consider the $\delta_2$-component of $\M_n(K)(0, \delta_2,\ldots,\delta_n).$ It is nonzero since the matrix unit $e_{21}$ is in it. The first row of any element of this component consists of zeros since $\delta_2+\delta_i>0$ and so $K_{-0+\delta_2+\delta_i}=0$ for all $i=1,\ldots, n.$ Hence, the determinant of any matrix in the $\delta_2$-component is zero and so 
$\M_n(K)(0, \delta_2,\ldots,\delta_n)_{\delta_2}$ does not contain an invertible element. By Lemma \ref{invertible_homogeneous_elements}, 
$\M_n(K)(0, \delta_2,\ldots,\delta_n)\cong_{\gr} \M_n(K)(\gamma_1, \gamma_2\ldots,\gamma_n)$ is not graded unit-regular.  

For the converse, note that the algebras $\M_n(K)(\gamma_1, \gamma_1,\ldots, \gamma_1)$ and $\M_n(K)(0, 0, \ldots, 0)$ are equal by part (2) of Lemma \ref{lemma_on_shifts}. The algebra $\M_n(K)(0, 0, \ldots, 0)$ is graded unit-regular since it is trivially graded and $\M_n(K)$ is unit-regular.

(2)  We prove direction $\Rightarrow$ by contrapositive. Assume that $m$ does not divide $n$ or that $n=km$ for some $k$ but that the list $\gamma_1,\ldots,\gamma_n,$ considered modulo $m,$ is such that some $i,j=0,1,\ldots, m-1$ appear different number of times. We show that there is a nonzero component of $M_n=\M_n(K[x^m, x^{-m}])(\gamma_1,\ldots,\gamma_n)$ such that all its elements have determinant zero and, consequently, are not invertible. By Lemma \ref{invertible_homogeneous_elements}, this shows that $M_n$ is not graded unit-regular. 

Using part (3) of Lemma \ref{lemma_on_shifts}, it is sufficient to consider the case $0\leq \gamma_j<m$ for $j=1, \ldots, n.$ Let $d_i$ be the number of times $i$ appears on the list $\gamma_1,\ldots,\gamma_n$ for all $i=0,\ldots, m-1.$ Since every $i=0,\ldots, m-1$ appears in the list by the assumption, $d_i>0$ for every $i.$ Let $j$ be such that $d_j=\min \{d_0, \ldots, d_{m-1}\}.$ Using part (2) of Lemma \ref{lemma_on_shifts}, we can add $m-j-1$ to all the entries of the list and use part (3) of Lemma \ref{lemma_on_shifts} again to consider the elements in the new list modulo $m$ again. By doing this, we can assume that $j=m-1.$ Permuting the entries using part (1) of Lemma \ref{lemma_on_shifts} and relabeling $d_0, \ldots, d_{m-1}$ if necessary, we can assume that $M_n$ is  
\[\M_n(K[x^m, x^{-m}])(0,0\ldots, 0,\; 1,1,\ldots, 1,\;\ldots\ldots\ldots,\; m-1,m-1,\ldots, m-1)\] where every $i$ appears $d_i$ times on the above list, $d_{m-1}\leq d_j$ for all $j=0,\ldots, m-1,$ and $d_{m-1}<d_i$ for at least one $i=0,\ldots, m-1.$  
The $(i+1)$-component is nonzero since the matrix unit $e_{s1}$ where $s=1+\sum_{j=0}^{i}d_i$ is in it. An arbitrary element $A$ of the $(i+1)$-component of $M_n$ can be divided into $m\times m$ blocks of sizes $d_j\times d_l$ for $j,l=0,\ldots, m-1.$ 
All of the blocks are zero except possibly 
\begin{center}
$d_0    \times d_{m-1-i},\; d_1  \times d_{m-i},\; \ldots,\; d_{i}\times   d_{m-1}\;\;$ and $\;\;
 d_{i+1}\times d_{0},\;    d_{i+2}\times d_{1},\;\ldots,\;  d_{m-1}\times d_{m-2-i}.$
\end{center}
Note that the block $d_{i}\times d_{m-1}$ is the only nonzero block in $d_{i}$ rows of $A$ and the last $d_{m-1}$ columns of $A$ and that there are more rows than columns in this block since $d_i>d_{m-1}.$ Compute the determinant of $A$ using expansion with respect to the row of $A$ corresponding to the first row of this block. Continue computing the minors of this minor. In each step, use the row corresponding to the first row of the remaining portion of this block. The condition $d_i>d_{m-1}$ implies that every minor of the $d_i-d_{m-1}$-th step is zero.  Thus, all the minors computed in the previous steps are zero also and, as a consequence, the determinant of $A$ is zero as well. 

To prove the converse, let $R=\M_m\left(\M_k(K[x^m, x^{-m}])(0, 0, \ldots, 0)\right)(0,1,\ldots,m-1).$ By Lemma \ref{lemma_on_shifts}, $R\cong_{\gr} \M_n(K[x^m, x^{-m}])(\gamma_1,\ldots,\gamma_n),$ so it is sufficient to show that $R$ is graded unit regular. 

An element $A$ of the $l$-component for $l=k'm+i'$ with $0\leq i'<m$  
can be written as $m\times m$ blocks of $k\times k$ matrices $\left[[a_{st}]_{ij}\right]$ with exactly $m$ possibly nonzero blocks. The blocks at the $(i'+1, 1), (i'+2, 2),\ldots (m, m-i')$ spots are elements of $$\M_k(K[x^m, x^{-m}])(0,0,\ldots, 0)_{k'm}=\M_k(K)x^{k'm}$$
and the blocks at $(1, m-i'+1), (2, m-i'+2), \ldots (i', m)$ spots are elements of $$\M_k(K[x^m, x^{-m}])(0,0,\ldots, 0)_{(k'+1)m}=\M_k(K)x^{(k'+1)m}.$$ 
For each $(i,j)\in \{(i'+1, 1), (i'+2, 2),\ldots (m, m-i')\},$  $a_{st}=b_{st}x^{k'm}\in Kx^{k'm}$  for all $s,t=1,\ldots, k.$ For 
such $(i,j),$ let $[v_{st}]_{ij}$ be an invertible matrix in $\M_n(K)$ such that $[b_{st}]_{ij} [v_{st}]_{ij}[b_{st}]_{ij}=[b_{st}]_{ij}$ and let $u_{st}=v_{st}x^{-k'm}$ for $s,t=1,\ldots,k.$
For each $(i,j)\in\{(1, m-i'+1), (2, m-i'+2), \ldots (i', m)\},$ $a_{st}=b_{st}x^{(k'+1)m}\in Kx^{(k'+1)m}$ for all $s,t=1,\ldots, k.$ For such $(i,j),$ let $[v_{st}]_{ij}$ be an invertible matrix in $\M_n(K)$ such that $[b_{st}]_{ij} [v_{st}]_{ij}[b_{st}]_{ij}=[b_{st}]_{ij}$ and let $u_{st}=v_{st}x^{-(k'+1)m}$ for $s,t=1,\ldots,k.$ For all other $(i,j),$ let $[u_{st}]_{ij}=0_{k\times k}.$ Finally, let $U_{ij}=[u_{st}]_{ij}$ and $U=\left[U_{ji}\right].$ 
By construction, $U$ is in the $-l$-component of $R_k$ and $AUA=A.$   
If $U_{ij}^{-1}$ is the inverse of $U_{ij}$ for $U_{ij}\neq 0$ and $U_{ij}^{-1}=0$ for $U_{ij}=0,$ then $U$ is invertible with $U^{-1}=[U_{ij}^{-1}].$  
\end{proof}

If $K$ is a  trivially $\Zset$-graded field,  Proposition \ref{proposition_matrices} can be used to readily conclude that UR$_{\gr}$ holds for 
$\M_9(K[x^3, x^{-3}])(0,0,0,1,1,1,2,2,2)$ and fails for 
$\M_9(K[x^3, x^{-3}])(0,0,0,0,1,1,1,2,2)$ for example. Proposition \ref{proposition_matrices} also generalizes the example from the introduction stating that $\M_2(K)(0,1)$ is not graded unit-regular and implies that $\M_n(K[x,x^{-1}])(\gamma_1, \ldots, \gamma_n)\cong_{\gr} \M_n(K[x,x^{-1}])(0, \ldots, 0)$ is graded unit-regular for any $n$ and any $\gamma_1, \ldots, \gamma_n\in \Zset.$

Proposition \ref{proposition_matrices} can also be used in the following example showing that graded unit-regularity is not necessarily passed to graded corners. 

\begin{example}\label{example_corners}
Let $R=\M_3(K[x^3, x^{-3}])(0,1,2)$ and $S=\M_2(K[x^3, x^{-3}])(0, 1).$ By Proposition \ref{proposition_matrices}, $R$ is graded unit-regular and $S$ is not graded unit-regular. Let $e=e_{11}+e_{22}$ where $e_{11}$ and $e_{22}$ are the standard matrix units. Then $eRe\cong_{\gr}S$ and so $eRe$ is not graded unit-regular. 
\end{example}

\subsection{Leavitt path algebras}
\label{subsection_LPAs}
We briefly review some relevant definitions. Let $E$ be a directed graph. The graph $E$ is row-finite if every vertex emits finitely many edges and it is finite if it has finitely many vertices and edges. A sink of $E$ is a vertex which does not emit edges. A vertex of $E$ is regular if it is not a sink and if it emits finitely many edges. A cycle is a closed path such that different edges in the path have different sources. A cycle has an exit if a vertex on the cycle emits an edge outside of the cycle. The graph $E$ is acyclic if there are no cycles. We say that graph $E$ is no-exit if $v$ emits just one edge for every vertex $v$ of every cycle. 

Let $E^0$ denote the set of vertices, $E^1$ the set of edges and $\so$ and $\ra$ denote the source and range maps of a graph $E.$ If $K$ is any field, the \emph{Leavitt path algebra} $L_K(E)$ of $E$ over $K$ is a free $K$-algebra generated by the set  $E^0\cup E^1\cup\{e^*\ |\ e\in E^1\}$ such that for all vertices $v,w$ and edges $e,f,$

\begin{tabular}{ll}
(V)  $vw =0$ if $v\neq w$ and $vv=v,$ & (E1)  $\so(e)e=e\ra(e)=e,$\\
(E2) $\ra(e)e^*=e^*\so(e)=e^*,$ & (CK1) $e^*f=0$ if $e\neq f$ and $e^*e=\ra(e),$\\
(CK2) $v=\sum_{e\in \so^{-1}(v)} ee^*$ for each regular vertex $v.$ &\\
\end{tabular}

By the first four axioms, every element of $L_K(E)$ can be represented as a sum of the form $\sum_{i=1}^n a_ip_iq_i^*$ for some $n$, paths $p_i$ and $q_i$, and elements $a_i\in K,$ for $i=1,\ldots,n.$ Using this representation, it is direct to see that $L_K(E)$ is a unital ring if and only if $E^0$ is finite in which case the sum of all vertices is the identity. For more details on these basic properties, see \cite{LPA_book}.

If we consider $K$ to be trivially graded by $\Zset,$ $L_K(E)$ is naturally graded by $\Zset$ so that the $n$-component $L_K(E)_n$ is  the $K$-linear span of the elements $pq^*$ for paths $p, q$ with $|p|-|q|=n$ where $|p|$ denotes the length of a path $p.$ While one can grade a Leavitt path algebra by any group $\Gamma$ (see \cite[Section 1.6.1]{Roozbeh_book}), we always consider the natural grading by $\Zset.$ 

By \cite[Proposition 5.1]{Roozbeh_Lia_Ultramatricial}, if $E$ is a finite no-exit graph, then  
$L_K(E)$ is graded isomorphic to 
\begin{equation}\label{representation} \tag{Rep}
R=\bigoplus_{i=1}^k \M_{k_i} (K)(\gamma_{i1}\ldots,\gamma_{ik_i}) \oplus \bigoplus_{j=1}^n \M_{n_j} (K[x^{m_j},x^{-m_j}])(\delta_{j1}, \ldots, \delta_{jn_j})
\end{equation}
where $k$ is the number of sinks, $k_i$ is the number of paths ending in the sink indexed by $i$ for $i=1,\ldots, k,$ and $\gamma_{il}$ is the length of the $l$-th path ending in the $i$-th sink for $l=1,\ldots, k_i$ and $i=1,\ldots, k.$ In the second term, $n$ is the number of cycles, $m_j$ is the length of the $j$-th cycle for $j=1,\ldots, n,$ $n_j$ is the number of paths which do not contain the cycle indexed by $j$ and which end in a fixed but arbitrarily chosen vertex of the cycle, and $\delta_{jl}$ is the length of the $l$-th path ending in the fixed vertex of the $j$-th cycle for $l=1,\ldots, n_j$ and $j=1,\ldots, n.$ This representation is not necessarily unique (see \cite[Example 2.2]{Lia_LPA_realization}), but it is unique up to a graded isomorphism by Lemma \ref{lemma_on_shifts}. We refer to the graded algebra $R$ above as a
{\em graded matricial representation} of $L_K(E).$

\subsection{Characterization of graded unit-regular Leavitt path algebras of finite graphs}\label{subsection_LPA_ur}
We use graded matricial representations and Proposition \ref{proposition_matrices} to prove the main result of this section.  
The notation EDL below is supposed to shorten ``equally distributed lenghts''.

\begin{theorem}
If $K$ is a field and $E$ is a finite graph, the following conditions are equivalent. 
\begin{enumerate}
\item $L_K(E)$ is graded unit-regular. 
\item $E$ is a no-exit graph without sinks which receive edges such that Condition (EDL) below holds.
\begin{itemize}
\item[(EDL)] For every cycle of length $m,$ the lengths, considered modulo $m$, of all paths which do not contain the cycle and which end in an arbitrary vertex of the cycle, are 
\[0,0,\ldots, 0,\; 1, 1, \ldots, 1,\; \ldots\ldots \ldots,\; m-1, m-1\ldots m-1 \]
where each $i$ is repeated the same number of times in the above list for $i=0, \ldots, m-1.$
\end{itemize}
\end{enumerate} 
\label{finite_graphs_ur_characterization} 
\end{theorem}
\begin{proof}
If (1) holds, then $L_K(E)$ is graded directly finite by Propositions \ref{graded_sr_and_gr_ur} and \ref{graded_sr_implies_df}. So, $E$ is a no-exit graph by \cite[Theorem 3.7]{Roozbeh_Ranga_Ashish}. Let $R$ be a graded matricial representation as in (\ref{representation}). Since $R$ is graded unit-regular, each graded direct summand of $R$ is graded unit-regular. If $k_i>1,$ then not all $\gamma_{i1}\ldots,\gamma_{ik_i}$ are equal since one of them is zero (corresponding to the trivial path of length zero to the $i$-th sink) and the others are positive (corresponding to the lengths of nontrivial paths to the $i$-th sink). So, Proposition \ref{proposition_matrices} and Lemma \ref{lemma_on_shifts} imply that $k_i=1$ for all $i=1,\ldots, k$ which means that the trivial path is the only one ending in the $i$-th sink for all $i=1,\ldots, k.$ 

For every $j=1,\ldots, n,$ each $l=0,\ldots m_j-1$ appears on the list $\delta_{j1}, \ldots, \delta_{jn_j}$ because there is a path of length $l$ which is a subpath of the $j$-th cycle and which ends at the selected vertex $v_j$ of the $j$-th cycle. Thus, Proposition \ref{proposition_matrices} and Lemma \ref{lemma_on_shifts} imply that $n_j$ is a multiple of $m_j$ and that the integers $\delta_{j1}, \ldots, \delta_{jn_j},$ considered modulo $m_j$ and permuted if necessary, produce a list as in Condition (EDL). Thus, the lengths, considered modulo $m_j,$ of paths which do not contain the $j$-th cycle and which end at $v_j$ are as listed in Condition (EDL).   

Conversely, assume that $E$ is such that (2) holds. Since $E$ is no-exit, a graded matricial representation $R$ of $L_K(E)$ has the form as in (\ref{representation}). By the assumption that no sink receives an edge, $k_i=1$ for every $i=1,\ldots, k.$ By the assumption that (EDL) holds, we can apply Lemma \ref{lemma_on_shifts} to permute the shifts and to replace each $\delta_{jl}$ by the remainder of the division by $m_j$ for $l=1,\ldots, n_j$ and $j=1,\ldots, n.$ This produces a graded isomorphism of $R$ and the algebra 
$$K^k\oplus \bigoplus_{j=1}^n\M_{k_jm_j} (K[x^{m_j},x^{-m_j}])(0, 0, \ldots, 0,\; 1,1,\ldots, 1,\; \ldots\ldots\ldots,\; m_j-1, m_j-1,\ldots, m_j-1)$$ 
where each $i=0,\ldots, m_j-1$ appears $k_j$ times in the list of shifts above for all $j=1,\ldots, n.$ By Proposition \ref{proposition_matrices}, every direct summand of this last algebra is graded unit-regular so $R$ is graded unit-regular also. Hence, $L_K(E)$ is graded unit-regular.  
\end{proof}

Theorem \ref{finite_graphs_ur_characterization} enables one to readily conclude that the Leavitt path algebras of the first two graphs below are not graded unit-regular while the opposite holds for the last two graphs.  
$$
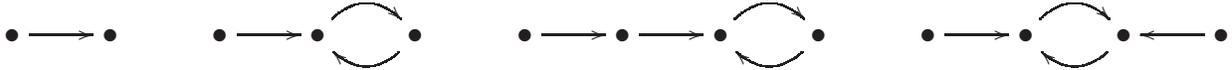
\xymatrix{ {\bullet} \ar[r]& {\bullet} }\hskip1cm
\xymatrix{ {\bullet} \ar[r]& {\bullet} \ar@/^1pc/ [r]   & {\bullet} \ar@/^1pc/ [l] }\hskip1cm
\xymatrix{ {\bullet} \ar[r]&{\bullet} \ar[r]& {\bullet} \ar@/^1pc/ [r]   & {\bullet} \ar@/^1pc/ [l]}  \hskip1cm \xymatrix{ {\bullet} \ar[r]& {\bullet} \ar@/^1pc/ [r]   & {\bullet} \ar@/^1pc/ [l] &\ar[l] {\bullet}}$$
Indeed, the first graph has a sink which receives an edge so condition (2) of Theorem \ref{finite_graphs_ur_characterization} fails. For the second graph,  0, 1, 1 are the lengths (modulo 2) of paths which end at any vertex of the cycle and which do not contain the cycle. So, (EDL) fails. 
For the last two graphs, 0, 0, 1, and 1 are the lengths (modulo 2) of the relevant paths and condition (2) of Theorem \ref{finite_graphs_ur_characterization} holds.    

\subsection{Characterizations of other cancellation properties of Leavitt path algebras}\label{subsection_LPA_other}

\begin{proposition}
Let $K$ be a field and $E$ be a graph such that $E^0$ is finite. For part (1), (2), (3) and (5), we also assume that $E^1$ is finite.  
\begin{enumerate}
\item The following conditions are equivalent.
\begin{enumerate}
\item $\sr_{\gr}(L_K(E))=1.$\hskip.4cm (b) IC$_{\gr}$ holds for $L_K(E).$\hskip.4cm (c) Condition (2) of Theorem \ref{finite_graphs_ur_characterization} holds.
\end{enumerate}  
\item  $\sr(L_K(E))=1$ if and only if $E$ is acyclic. 

\item  $L_K(E)$ is graded weakly unit-regular if and only if $E$ is no-exit. 

\item  IC holds for $L_K(E)$ if and only if $E$ is no-exit.  

\item Conditions Reg$_{\gr}$+UR$_\epsilon$ and S$_{\gr}(L_K(E))$ holds for $L_K(E).$   
\end{enumerate}
\label{LPA_characterizations}
\end{proposition}
\begin{proof}
Note that the assumption that $E^0$ is finite ensures that $L_K(E)$ is unital. The assumption that $E^1$ is also finite in part (1) enables us to use Theorem \ref{finite_graphs_ur_characterization} and in parts (2) and (3) ensures that a graded matricial representation of a no-exit graph has the form as in Section \ref{subsection_LPAs}.   

(1) Since $L_K(E)$ is graded regular (\cite[Theorem 9]{Roozbeh_regular}), Propositions \ref{graded_sr_and_gr_ur}, \ref{graded_unit_regular}, and  Theorem \ref{finite_graphs_ur_characterization} imply that each of (a), (b) and (c) is equivalent with the condition that UR$_{\gr}$ holds for $L_K(E).$   

(2) If $\sr(L_K(E))=1,$ then $L_K(E)$ is directly finite by Proposition \ref{graded_sr_implies_df} in the nongraded case. Thus, $E$ is no-exit by \cite[Theorem 4.12]{Lia_traces} and so $L_K(E)$ is isomorphic to a direct sum of matricial algebras over $K$ and $K[x,x^{-1}]$ (which is a matricial representation if we ignore the grading). Since $\sr(K[x,x^{-1}])>1$ and $\sr(L_K(E))=1,$ there cannot be matrix algebras over $K[x,x^{-1}]$ present. Thus, $E$ is acyclic. Conversely, if $E$ is acyclic, then $L_K(E)$ is unit-regular by \cite[Theorem 2]{Gene_Ranga_regular} and so $\sr(L_K(E))=1$ by Proposition \ref{graded_sr_and_gr_ur} in the nongraded case. 

(3) If $L_K(E)$ is graded weakly unit-regular, IC$^{\w}_{\gr}(L_K(E))$ holds by Proposition \ref{graded_unit_regular}. The condition IC$^{\w}_{\gr}(L_K(E))$ implies DF$^{\s}_{\gr}(L_K(E))$ as we showed in Section \ref{subsection_graded_finiteness}. 
By the assumption that $E^0$ is finite, $L_K(E)$ is unital and so $L_K(E)\cong_{\gr}\End_{L_K(E)}(L_K(E)).$ Thus, the condition DF$^{\s}(L_K(E))$ implies that $L_K(E)$ is graded directly finite. By \cite[Theorem 3.7]{Roozbeh_Ranga_Ashish}, $E$ is a no-exit graph. Conversely, if $E$ is a no-exit graph, then a graded matricial representation of $L_K(E)$ is graded semisimple, and hence weakly graded unit-regular. Here we make use of the fact that the implication ``semisimple $\Rightarrow$ UR'' directly implies ``graded semisimple $\Rightarrow$ UR$^{\w}_{\gr}$''. Thus, $L_K(E)$ is weakly graded unit-regular. 

(4) If IC$(P)$ holds for every finitely generated projective $L_K(E)$-module $P$, then IC$(L_K(E))$ holds. 
Since  IC$(\underline{\hskip.3cm})\Rightarrow$ DF$(\underline{\hskip.3cm}),$ DF$(L_K(E))$ holds. Using the same argument as in the proof of (3), the assumption that $E^0$ is finite ensures that the condition DF($L_K(E)$) implies that $L_K(E)$ is directly finite. By \cite[Theorem 4.12]{Lia_traces}, $E$ is a no-exit graph. Conversely, if $E$ is no-exit, then $L_K(E)$ is cancellable by \cite[Lemma 5.5]{Ara_et_al_Steinberg}. So, C$(P)$ holds for every finitely generated projective $L_K(E)$-module $P$ which implies that IC$(P)$ holds for every such module $P.$ 

(5) Since $E$ is finite, the algebra $L_K(E)_0$ is an ultramatricial algebra over $K$ (see \cite[Corollary 2.1.16]{LPA_book}) with unital connecting maps. So, $L_K(E)_0$ is unit-regular and $\sr(L_K(E)_0)=1.$ Hence, UR$_\epsilon$ holds and S$_{\gr}(L_K(E))$ holds by Theorem \ref{substitution_and_sr}. Reg$_{\gr}$ holds by  \cite[Theorem 9]{Roozbeh_regular}.
\end{proof}

The diagram below summarizes the statements above for finite graphs.  
{\small
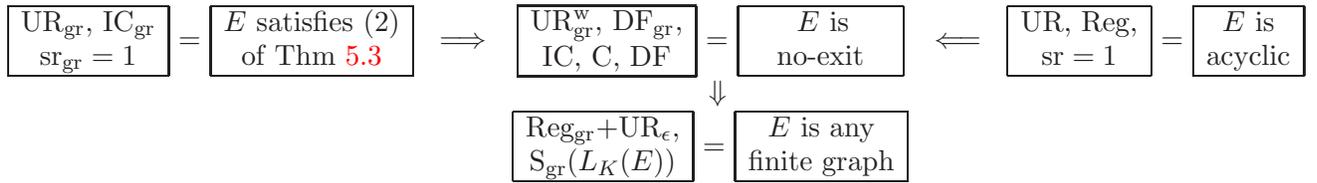
\begin{center}
\begin{tabular}{ccccc}
\begin{tabular}{|c|} \hline 
UR$_{\gr},$ IC$_{\gr}$\\ $\sr_{\gr}=1$\\ \hline
\end{tabular}
= 
\begin{tabular}{|c|} \hline 
$E$ satisfies (2)  \\ of Thm \ref{finite_graphs_ur_characterization} \\ \hline
\end{tabular}
& $\Longrightarrow$ & 
\begin{tabular}{|c|} \hline 
UR$^{\w}_{\gr},$ DF$_{\gr}$,\\ IC, C, DF\\ \hline
\end{tabular}  = 
\begin{tabular}{|c|} \hline 
$\,E$ is \\ $\;\;$ no-exit $\;\;$ \\ \hline
\end{tabular}& $\Longleftarrow$& 
\begin{tabular}{|c|} \hline 
UR, Reg,\\ $\sr=1$ \\ \hline
\end{tabular} 
= \begin{tabular}{|c|} \hline 
$E$ is \\ acyclic \\ \hline
\end{tabular} \\ 
&&$\hskip.2cm\Downarrow$ &&\\
&&\begin{tabular}{|c|} \hline 
Reg$_{\gr}+$UR$_\epsilon,$\\ S$_{\gr}(L_K(E))$ \\ \hline
\end{tabular} =
\begin{tabular}{|c|} \hline 
$E$ is any \\ finite graph  \\ \hline
\end{tabular}&&\\
\end{tabular}
\end{center}}

\subsection{Possible generalizations}

A local version of a ring-theoretic property $P$ is typically obtained by requiring that for every finite set $F,$ there is an idempotent $e$ such that $F\subseteq eRe$ and $eRe$ has property $P.$ If $R$ is non-unital, this definition enables one to consider local versions of properties whose definitions require the existence of the ring identity.

The properties of being unit-regular and directly finite can be generalized to non-unital rings in this way. This approach has been used in \cite{Gene_Ranga_regular} for unit-regularity and in \cite{Lia_traces} for direct finiteness. While the condition $(\forall a,b)(aR+bR=R \Rightarrow (\exists x) (a+bx)R=R)$
does not specifically include the identity, it is just a shorter version of the condition $(\forall a,b)((\exists c,d) ac+bd=1 \Rightarrow (\exists x, u) (a+bx)u=1)$ where the identity does appear. So, $\sr(R)=1$ should also be treated as a property of unital rings. 

In the graded case, properties of unital graded rings can be generalized to non-unital case in the same way. In particular, a graded, possibly non-unital, ring $R$ is {\em graded locally unit-regular} if  for every finite set $F$, there is a homogeneous idempotent $u$ such that $F\subseteq uRu$ and $uRu$ is graded unit-regular. A graded ring having graded locally stable range 1, a graded locally directly finite ring, and a graded locally weakly unit-regular ring can be defined analogously. 

Using these definitions, it is possible to consider graded local cancellability properties of Leavitt path algebras over graphs without any restrictions on their cardinality. Given this fact, we wonder whether the requirements that $E$ is finite can be dropped from the results of sections \ref{subsection_LPA_ur} and \ref{subsection_LPA_other}. In particular, we wonder about the following. 

\begin{question}
What graph-theoretic condition is equivalent to the condition that the Leavitt path algebra of an arbitrary graph is graded locally unit-regular?  
\label{question_graded_ur}
\end{question}

Graded regularity passes to graded corners so the local version of Proposition \ref{graded_sr_and_gr_ur} holds. Hence, an answer to the above question would provide characterization of graded locally stable range 1 also because every Leavitt path algebra is graded regular.

\subsection{More on Question \ref{question_graded_Handelman}}\label{subsection_Handelman}
As mentioned in the introduction, considering the graded version of Handelman's Conjecture provides further evidence that Reg$_{\gr}$+UR$_\epsilon$ is more suited as a graded analogue of UR than UR$_{\gr}.$ Recall that Handelman's Conjecture states that a ring with involution which is $*$-regular (see \cite{Berberian_web} or \cite{Gonzalo_Ranga_Lia} for definition and basic properties) is necessarily directly finite and unit-regular. While the part on direct finiteness has been shown to hold, the part on unit-regularity is still open. In \cite{Gonzalo_Ranga_Lia}, the authors note that this conjecture holds for Leavitt path algebras. In \cite{Roozbeh_Lia_Baer}, the authors consider the graded version of $*$-regularity and note that every Leavitt path algebra over a field $K$ with a positive definite involution (for any $n$ and any $k_1,\dots,k_n\in K$, $\sum_{i=1}^n k_ik_i^* = 0$ implies $k_i=0$ for $i=1,\ldots,n$) is graded $*$-regular. They also note that if $E$ is the graph from part (3) of Example \ref{example_cancellability_conditions}, then $L_K(E)$ is not graded unit-regular for any $K$ so the graded version of Handelman's Conjecture fails. However, as we argued in the introduction, Question \ref{question_graded_Handelman} is more relevant as a graded version of Handelman's Conjecture.  For unital Leavitt path algebras, the answer to this question is ``yes'' since unital Leavitt path algebras satisfy Reg$_{\gr}$+UR$_\epsilon$ by Proposition \ref{LPA_characterizations}.


\begin{thebibliography}{10}
\bibitem{LPA_book} G. Abrams, P. Ara, M. Siles Molina, Leavitt path algebras, Lecture Notes in Mathematics 2191, Springer, London, 2017. 

\bibitem{Gene_Ranga_regular} G. Abrams, K. M. Rangaswamy, \emph{Regularity conditions for arbitrary Leavitt path algebras,} Algebr. Represent. Theory  \textbf{13 (3)} (2010), 319--334.

\bibitem{Gonzalo_Ranga_Lia} G. Aranda Pino, K. M. Rangaswamy, L. Va\v s, \emph{$^\ast$-regular Leavitt path algebra of arbitrary graphs}, Acta Math. Sci. Ser. B  Engl. Ed.  \textbf{28 (5)} (2012), 957 -- 968.

\bibitem{Ara_et_al_Steinberg} P. Ara, R. Hazrat, H. Li, A. Sims, \emph{Graded Steinberg algebras and their representations}, Algebra Number Theory {\bf 12 (1)} (2018), 131--172.

\bibitem{Berberian_web} S. K. Berberian, \emph{Baer rings and Baer $*$-rings}, 1988, preprint at \url{https://www.ma.utexas.edu/mp_arc/c/03/03-181.pdf}.

\bibitem{Goodearl_book} K. R. Goodearl, von Neumann regular rings, 2nd ed., Krieger Publishing Co., Malabar, FL, 1991.

\bibitem{Roozbeh_regular} R. Hazrat, \emph{Leavitt path algebras are graded von Neumann regular rings,} J. Algebra {\bf 401} (2014), 220--233.

\bibitem{Roozbeh_book} R. Hazrat, Graded rings and graded Grothendieck groups, London Math. Soc. Lecture Note Ser. 435, Cambridge Univ. Press, 2016.

\bibitem{Roozbeh_Ranga_Ashish} R. Hazrat, K. M. Rangaswamy, A. K. Srivastava, \emph{Leavitt path algebras: graded direct-finiteness and graded $\sum$-injective simple modules}, J. Algebra {\bf 503} (2018), 299--328.

\bibitem{Roozbeh_Lia_Baer}  R. Hazrat, L. Va\v s, \emph{Baer and Baer $*$-ring characterizations of Leavitt path algebras}, J. Pure  Appl. Algebra, {\bf 222 (1)} (2018), 39 -- 60.

\bibitem{Roozbeh_Lia_Ultramatricial}  R. Hazrat, L. Va\v s, \emph{$K$-theory classification of graded ultramatricial algebras with involution}, Forum Math., {\bf 31 (2)} (2019), 419--463. 

\bibitem{Lam_Murray} T. Y. Lam, W. Murray, \emph{Unit regular elements in corner rings}, Bull. Hong Kong Math. Soc. {\bf 1 (1)} (1997), 61--65. 

\bibitem{Lam_cancellation_properties} T. Y. Lam, \emph{A crash course on stable range, cancellation, substitution and exchange}, J. Algebra Appl. {\bf 03} (2004), 301--343.

\bibitem{NvO_book}  C. N\u ast\u asescu, F. van Oystaeyen, Methods of graded rings, Lecture Notes in Mathematics 1836, Springer-Verlag, Berlin, 2004.

\bibitem{Lia_traces} L. Va\v s, \emph{Canonical traces and directly finite Leavitt path algebras}, Algebr. Represent. Theory {\bf 18} (2015), 711--738. 

\bibitem{Lia_realization} L. Va\v s, \emph{Simplicial and dimension groups with group action and their realization}, submitted for publication, preprint {\tt arXiv: 1805.07636 [math.KT]}. %check

\bibitem{Lia_LPA_realization} L. Va\v s, \emph{Realization of graded matrix algebras as Leavitt path algebras}, Beitr. Algebra Geom., in print, preprint {\tt arXiv:1910.05174 [math.RA]}.  %check

\bibitem{Vaserstein1} L. N. Vaserstein, \emph{Stable rank of rings and dimensionality of topological spaces}, Funktsional. Anal. i Prilozhen. {\bf 5} (1971), 17--27. English translation: Funct. Anal. Appl., {\bf 5} (1971), 102--110.

\bibitem{Vaserstein2} L. N. Vaserstein, \emph{Bass's first stable range condition}, J. Pure Appl. Algebra {\bf 34} (1984), 319--330.

\end{thebibliography}
\end{document}